\documentclass{amsart}

\usepackage[top=2cm, bottom=2.8cm, left=2cm, right=2cm]{geometry}
\usepackage[english]{babel}
\usepackage{amsmath} 
\usepackage{amsthm}
\usepackage{amsfonts}
\usepackage{amssymb}
\usepackage{graphicx}
\usepackage{mathrsfs}
\usepackage{todonotes}
\usepackage{a4wide,color,eucal,enumerate,mathrsfs}
\usepackage[normalem]{ulem}
\usepackage{amsmath,amssymb,epsfig,bbm}
\usepackage{pdfsync}
\usepackage{tikz-cd}
\numberwithin{equation}{section}
\usepackage[pdfborder={0 0 0}]{hyperref}  
\usepackage{hyperref}\hypersetup{colorlinks=true}

\newcommand{\N}{\mathbb{N}}
\newcommand{\Q}{\mathbb{Q}}
\newcommand{\R}{\mathbb{R}}

\newcommand{\MM}{\mathscr{M}}

\newcommand{\PP}{\mathscr{P}}

\newcommand{\cA}{{\ensuremath{\mathcal A}}}
\newcommand{\cB}{{\ensuremath{\mathcal B}}}

\newcommand{\cF}{{\ensuremath{\mathcal F}}}

\newcommand{\cI}{{\ensuremath{\mathcal I}}}

\newcommand{\cM}{{\ensuremath{\mathcal M}}}

\newcommand{\cP}{{\ensuremath{\mathcal P}}}

\newcommand{\cW}{{\ensuremath{\mathcal W}}}

\newcommand{\mm}{{\mbox{\boldmath$m$}}}

\newcommand{\bbeta}{{\mbox{\boldmath$\beta$}}}
\newcommand{\ggamma}{{\mbox{\boldmath$\gamma$}}}

\newcommand{\ppi}{{\mbox{\boldmath$\pi$}}}

\newcommand{\sfd}{{\sf d}}
\newcommand{\sfe}{{\sf e}}

\newcommand{\sfC}{{\sf C}}

\newcommand{\restr}[1]{\lower3pt\hbox{$|_{#1}$}}

\newcommand{\eps}{\varepsilon}  
\newcommand{\nchi}{{\raise.3ex\hbox{$\chi$}}}
\newcommand{\weakto}{\rightharpoonup}

\newcommand{\Probabilities}[1]{\mathscr P(#1)}          
\newenvironment{comment}
{\begin{quote}
       \rule{.5\textwidth}{.5pt}\ \\
       \footnotesize}
{\\\rule{.5\textwidth}{.5pt}
\end{quote}}

\newtheorem{theorem}{Theorem}[section]

\newtheorem{corollary}[theorem]{Corollary}
\newtheorem{lemma}[theorem]{Lemma}
\newtheorem{proposition}[theorem]{Proposition}
\newtheorem{definition}[theorem]{Definition}

\newtheorem{remark}[theorem]{Remark}

\newcommand{\Lip}{\mathrm{Lip}}
\newcommand{\lip}{\mathrm{lip}}

\newcommand{\fr}{\hfill$\blacksquare$}   
\newcommand{\rcd}{\mathrm{RCD}}
\newcommand{\cd}{\mathrm{CD}}
\newcommand{\mcp}{\mathrm{MCP}}
\renewcommand{\mm}{\mathfrak m}

\newcommand{\lims}{\varlimsup}
\newcommand{\limi}{\varliminf}
\renewcommand{\limsup}{\varlimsup}
\renewcommand{\liminf}{\varliminf}
\renewcommand{\d}{{\rm d}}

\newcommand{\Geo}{{\sf Geo}}

\newcommand{\X}{{\rm X}}
\newcommand{\Y}{{\rm Y}}
\newcommand{\Z}{{\rm Z}}

\newcommand{\Lop}{L^0_p(T^*\X)}
\newcommand{\e}{\sfe}
\newcommand{\msp}{|\dot{\gamma}_t|}
\newcommand{\sfKe}{{\sf Ke}}
\newcommand{\Keq}{\sfKe_q}
\newcommand{\rmKe}{{\rm Ke}}
\newcommand{\OptGeo}{{\rm OptGeo}}
\newcommand{\Opt}{{\rm Opt}}
\newcommand{\Adm}{\Pi}
\newcommand{\Xdm}{(\X,\sfd,\mm)}
\newcommand{\Xdmx}{(\X,\sfd,\mm,x)}
\newcommand{\Xdmxinf}{(\X_\infty,\sfd_\infty,\mm_\infty,x_\infty)}
\newcommand{\Xdmxn}{(\X_n,\sfd_n,\mm_n,x_n)}
\newcommand{\Comp}{{\rm Comp}}

\newcommand{\supp}{{\rm supp}}

\usepackage{ esint }
\usepackage{tcolorbox}

\title{A first-order condition for the independence on $p$ of weak gradients}
\author{Nicola Gigli}
\address{SISSA, Via Bonomea 265, 34136 Trieste}
\email{ngigli@sissa.it}
\author{Francesco Nobili}
\email{francesco.f.nobili@jyu.fi}

\begin{document}
\maketitle

\begin{abstract} It is well known that on arbitrary metric measure spaces, the notion of minimal $p$-weak upper gradient may depend on $p$.

In this paper we investigate how a first-order condition of the metric-measure structure, that we call Bounded Interpolation Property, guarantees that in fact such dependence is not present. 

We also show that the Bounded Interpolation Property  is stable for pointed measure Gromov Hausdorff convergence and holds on  a large class of spaces satisfying  curvature dimension conditions.
\end{abstract}

\tableofcontents

\bigskip

\section{Introduction}
In doing analysis on metric measure spaces one is often led to  consider the class of real valued Sobolev functions $W^{1,p}(\X)$, $p\in(1,\infty)$. There are various possible definitions for this object (\cite{H96}, \cite{Cheeger00}, \cite{Shanmugalingam00}, \cite{AmbrosioGigliSavare11}), most of them being equivalent: in this paper we are concerned with the approach proposed by Cheeger  \cite{Cheeger00}, which is equivalent to the one adopted by Shanmughalingham \cite{Shanmugalingam00} and to the more recent ones proposed by Ambrosio, Savar\'e and the first author \cite{AmbrosioGigliSavare11}, \cite{AmbrosioGigliSavare11-3}.

A common feature of all these equivalent definitions is that a Sobolev function $f\in W^{1,p}(\X)$  comes with a non-negative function in $L^p(\X)$ playing the role of the `modulus of the distributional differential'.  Unlike the Euclidean setting, this auxiliary function -  called `$p$-minimal weak upper gradient' and denoted by $|D f|_p$ -  might depend on $p$. More precisely, a quite direct consequence of the definitions is that for $p_1,p_2\in(1,\infty)$, $p_1<p_2$ we have $ W^{1,p_2}_{loc}(\X)\subset  W^{1,p_1}_{loc}(\X)$ and
\begin{equation}
|Df|_{p_1} \le |Df|_{p_2} \qquad \mm\text{-a.e.}\qquad\forall f\in W^{1,p_2}_{loc}(\X). \label{eq:Wp1p2}
\end{equation}
It is well known that in general this is all one can say. More precisely:
\begin{itemize}
\item[i)] In \cite{DiMarinoSpeight13}, a detailed study on weighted Euclidean spaces has been performed to build a family of metric measure spaces for which the strict inequality may occur in \eqref{eq:Wp1p2} (earlier examples in this direction have been proposed in \cite{KKST12}, see also the discussion in \cite[Section 2.9]{Bjorn-Bjorn11});
\item[ii)] There are examples (see, e.g.\ \cite[Remark 5.2]{KKST12} and \cite[Pag. 18]{AmbrosioGigliSavare11-3}) of metric measure spaces $\X$ showing that we may have:
\[ f \in W^{1,p_1}(\X), \ f,|Df|_{p_1} \in L^{p_2}(\mm) \quad \text{ but } \quad f \notin W^{1,p_2}(\X), \quad \forall p_2>p_1.\]
\end{itemize}
On the other hand, for very regular structures (e.g.\ the standard Euclidean space) neither i) nor ii) can occur, thus one can wonder if there are regularity assumptions one can impose on the given metric measure structure to ensure that even in such more abstract setting i), ii) do not occur. Some result in this direction is already known:
\begin{itemize}
\item[-] In \cite{Cheeger00} it has been proved that on doubling spaces supporting a Poincar\'e inequality, for Lipschitz functions $f$ we have that  $|Df|_p$ coincides a.e.\ with the local Lipschitz constant. A quite direct consequence of this fact is that i) cannot occur (see e.g.\ \cite[Corollary A.9]{Bjorn-Bjorn11} - we shall report this fact in Theorem \ref{prop:Chdoubpoincare}). In \cite[Lemma 5.1]{KKST12} it has been proved that, under the same assumptions, the situation ii) cannot occur, either.
\item[-] In  \cite{GigliHan14} it has been proved that on $\rcd(K,\infty)$ spaces neither i) nor ii) can occur, with a technique based on heat flow regularization and the Bakry-\'Emery inequality. Notice that $\rcd(K,\infty)$ spaces are not doubling in general.
\end{itemize} 
The goal of this paper is to further investigate this topic, our main results being:
\begin{itemize}
\item[a)] To propose an axiomatization of spaces having $p$-independent weak upper gradients (Definition \ref{def:pindependent}), especially distinguishing between a \emph{weak} and \emph{strong} kind of independence, roughly consisting in asking that i) cannot occur or that neither i) nor ii) can occur, respectively.
\item[b)] To investigate the consequences of these independences for what concerns the abstract differential calculus as proposed in \cite{Gigli14}. Perhaps not surprisingly, if weak upper gradients are independent of $p$, then the corresponding differentiation operators are also independent of $p$ (see Theorems \ref{thm:dpindepenw}, \ref{thm:dpindepens} and notice that in principle for any Sobolev exponent $p$ one has an abstract differential $\d_p$ that might depend on $p$).
\item[c)] To single out a regularity property yielding $p$-independent weak upper gradients in a strong sense (Definition \ref{def:bip} and Theorem \ref{thm:pind}). The property we introduce, that we shall call Bounded Interpolation Property, is a first order property (in particular, it remains valid if we replace the original measure $\mm$ with  $\rho\mm$ with $\rho,\frac1\rho\in L^\infty$). With respect to the previous available results recalled above, we obtain an independence  without relying on any local doubling property, and thus valid also in certain genuine infinite-dimensional situations.
\item[d)] We prove that the Bounded Interpolation Property is stable under measured Gromov Hausdorff convergence (Theorem \ref{thm:pmGHBIP}).
\item[e)] We show that a large class of spaces satisfying  a curvature dimension condition possess the Bounded Interpolation Property. In particular, this is the case for non branching $\mcp(K,N)$ spaces (Theorem \ref{thm:BIPMCP}).
\end{itemize}

\section{Preliminaries}
\begin{quote}
\begin{center}\textbf{Standing notation}\end{center} \noindent Throughout this note, it goes without saying that, whenever we fix $p \in [1,\infty)$, the letter $q$, even if not introduced, is automatically defined as the \emph{conjugate exponent} by the relation
\[ \tfrac1p +\tfrac1q =1\] and thus ranging in $(1,\infty]$ (with the usual extreme case $p=1$ and $q=\infty$). The converse convention is also kept, i.e. if we start fixing a number $q$, then automatically $p$ is defined as above. Typical recurrent exponents will be $p_1,p_2, \bar p$, thus giving rise to the conjugate numbers $q_1,q_2,\bar q$.
\end{quote} 

\noindent Let us fix $(\X,\sfd)$ a complete and separable metric space and $q \in (1,\infty)$. We equip the space of continuous path on $[0,1]$ with values   in $\X$, denoted by $C([0,1],\X)$, with the uniform distance $\sfd_{\text{sup}}(\gamma,\eta):= \sup_{t\in[0,1]} \sfd(\gamma_t,\eta_t)$, for every $\gamma,\eta \in C([0,1],\X)$. This turn $C([0,1],\X)$ into a complete and separable metric space. Next, we recall that the set of $q$-absolutely continuous curves, denoted by AC$^q([0,1],\X)$, is the subset of $\gamma \in C([0,1],\X)$ so that there exists $g \in L^q(0,1)$ satisfying
\[ \sfd(\gamma_t,\gamma_s) \le \int_s^t g(r)\, \d r, \qquad \forall s\le t \text{ in } [0,1].\]  
We recall that, for any $\gamma \in $ AC$^q([0,1],\X)$, there exists a minimal a.e. function $g \in L^q(0,1)$ satisfying the above, called metric speed and denoted $\msp$, which is defined as $\msp:= \lim_{h\downarrow 0} \sfd(\gamma_{t+h},\gamma_t)/h$, for a.e. $t$. Then, define the Kinetic energy functional $C([0,1],\X) \ni \gamma \mapsto \Keq(\gamma) := \int_0^1 | \dot{\gamma}_t|^q\, \d t$, if $\gamma \in AC^q([0,1],\X)$, $+\infty$ otherwise. We recall the following well known lemma and, given that we are going to deal frequently with the objects below, we also provide a simple proof.
\begin{lemma}\label{lem:closuremsp}
Let $(\X,\sfd)$ be a metric space, $q \in (1,\infty)$ and $(\gamma^n) \subseteq $ \emph{AC}$^q([0,1],\X)$ uniformly converging to $\gamma \in C([0,1],\X)$ with $\sup_n \Keq(\gamma^n) < \infty$. 

Then, $\gamma \in \emph{AC}^q([0,1],\X) $ and $ \Keq(\gamma)\le \liminf_n \Keq(\gamma^n).$ Moreover, if $\Keq(\gamma^n) \rightarrow \Keq(\gamma)$ (i.e. there is conservation of the Kinetic Energy), then one also recovers $|\dot{\gamma}_{\cdot}^n| \rightarrow |\dot{\gamma}_{\cdot}|$ in $L^q(0,1)$.
\end{lemma} 
\begin{proof}
For the first part, it suffices to notice that any weak-$L^q$ limit $G$ (possibly along a not relabeled suitable subsequence) of $|\dot{\gamma}_{\cdot}^n|$ in $L^q(0,1)$ satisfies
\[ \sfd(\gamma_t,\gamma_s) = \lim_n\sfd(\gamma^n_t,\gamma^n_s) \le \int_s^tG(r)\, \d t, \qquad \forall s,t \in [0,1].\]
Thus, $\gamma \in $ AC$^q([0,1],\X)$ and, by minimality of $\msp$ and weak lower semicontinuity of $L^q$-norms, one has
\[\Keq(\gamma) \le \int_0^1G^q(t)\, \d t \le \liminf_{n\rightarrow \infty}\Keq(\gamma^n)\]
Moreover, under the hypotheses of the second claim, the above becomes a chain of equalities ensuring that $G =\msp$ is a strong limit in $L^q(0,1)$.
\end{proof}

Next, we will make use of the so-called evaluation map at time $t \in [0,1]$, $\e_t \colon C([0,1],\X)\rightarrow \X$ defined via $\e_t(\gamma):= \gamma_t$. 

In this note, a metric measure space  is a triple $\Xdm$ where
\[\begin{split}
(\X,\sfd)\qquad & 	\text{is a complete and separable metric space},\\
\mm\neq 0\qquad & \text{is non negative and boundedly finite Borel measure}.
\end{split}
\]
We denote by  $\cM_b^+(\X)$ the space of finite Borel positive measures over $\X$. Also, we write $\cP(\X)$ for the space of probability measures over $\X$ which here will be often equipped with the \emph{weak topology} in duality with  $C_b(\X)$, the space of continuous and bounded functions. In this case, convergence of $\mu_n$ to $\mu$ will be written $\mu_n\weakto \mu$. Recall that, if $\phi \colon \X \rightarrow \Y$ is a Borel map between two metric spaces and $\mu \in \cP(\X)$, the set-value map $\cB(\Y) \ni B\mapsto \phi_\sharp \mu(B):= \mu(\phi^{-1}(B))$ is called the \emph{pushforward} measure of $\mu$ via $\phi$, and clearly $\phi_\sharp\mu \in \cP(\Y)$. When $\phi$ is continuous, the operation $\phi_\sharp$ is \emph{weakly} continuous. A family $\mathcal{K}\subset \cP(\X)$ is called \emph{tight}, provided
\[ \forall \epsilon>0, \exists K_\epsilon \subset \X \text{ compact so that} \quad \mu(\X\setminus K_\epsilon) \le \epsilon, \qquad \forall \mu \in \mathcal{K}.\]
For later use, we report without proof a well known characterization of compactness in the weak topology.
\begin{theorem}[Prokhorov]\label{thm:prokh}
Let $(\X,\sfd)$ be a complete and separable metric space and $\mathcal{K}\subset \cP(\X)$. The following are equivalent:
\begin{itemize}
\item[{\rm i)}] $\mathcal{K}$ is precompact in the weak topology;
\item[{\rm ii)}] $\mathcal{K}$ is tight;
\item[{\rm iii)}] There exists a functional $\psi\colon \X\to [0,\infty]$ with compact sublevels so that
\[ \sup_{\mu \in \mathcal{K}}\int \psi \, \d \mu <\infty.\]
\end{itemize}
\end{theorem}
Finally, we shall denote as usual by $\Lip(\X)$ and  $\Lip_{bs}(\X)$, the spaces of Lipschitz and boundedly supported Lipschitz functions , respectively. Moreover, if $f \in \Lip(\X)$, we call $\Lip(f)$ its Lipschitz constant.

\subsection{$q$-test plans and Sobolev space} We follow the definition of Sobolev spaces proposed in \cite{AmbrosioGigliSavare11-3} (for earlier approaches see the original work \cite{Cheeger00} of Cheeger and the one \cite{Shanmugalingam00} of Shanmugalingam). 
\begin{definition}[$q$-test plan]\label{def:qplan}
Let $\Xdm$ be a metric measure space and $q \in (1,\infty)$. A measure $\pi \in \cP(C(\X,[0,1]))$ is said to be a $q$\emph{-test plan}, provided 
\begin{itemize}
\item[{\rm i)}] there exists $C>0$ so that $(\e_t)_{\sharp}\pi \le C\mm$ for every $t \in [0,1]$;
\item[{\rm ii)}] we have $\int \Keq(\gamma) \, \d \pi <\infty$.
\end{itemize} 
Moreover, we say that $\pi$ is an $\infty$\emph{-test plan} if, instead of {\rm ii)}, we require 
\begin{itemize}
\item[{\rm ii')}] $\pi$ is concentrated on equi-Lipschitz curves, i.e.\ for some $L>0$ we have  $\Lip(\gamma)\leq L$ for $\pi$-a.e.\ $\gamma$ (and thus for every $\gamma$ in the support of $\pi$ by the lower-semicontinuity of the global Lipschitz constant w.r.t.\ uniform convergence).
\end{itemize}
\end{definition}
We usually refer to {\rm i)} as the `compression condition' and denote by $\Comp(\pi)$ the smallest constant $C$ satisfying {\rm i)} (and call it compression constant of $\pi$). The (full) kinetic energy of a $q$-test plan $\pi$, which we denote by $\rmKe_q(\pi)$, is simply $\|\Keq\|_{L^1(\pi)}$. Also, notice that
\begin{equation}
\pi \mapsto \rmKe_q(\pi) \qquad \text{is weakly lower semicontinuous,}\label{eq:KEqlsc}
\end{equation} 
since the integrand $\Keq(\gamma)$ appearing in the definition of $\rmKe_q(\cdot)$ is a lower semicontinuous on $C([0,1],\X)$ by Lemma \ref{lem:closuremsp} and bounded from below, therefore it admits the representation as supremum of continuous and bounded functions. Moreover, we shall sometimes deal with plans $\pi$ having bounded support: this is equivalent to say that $\{\gamma_t \colon \gamma \in \text{supp}(\pi), \ t \in  [0,1]\}\subset \X$ is bounded. In this note, it will be important to notice that the class of $q$-test plans is closed under restriction and rescaling. Namely, for any $q$-test plan $\pi$, we have:
\begin{itemize}
\item[$\circ$] if $\Gamma \subset C([0,1],\X)$ is Borel and $\pi(\Gamma) >0$ then, the \emph{rescaling} 
\begin{equation}
\frac{\pi\restr{\Gamma}}{\pi(\Gamma)}\qquad  \text{is a } q\text{-test plan}.\label{eq:pirescal}
\end{equation}
\item[$\circ$] let $s,t \in [0,1]$ with $s\le t$ and $\mathsf{Restr}_s^t \colon C([0,1],\X)\rightarrow C([0,1],\X)$ be the map via $\mathsf{Restr}_s^t(\gamma) := \gamma_{(1-\cdot)s + \cdot t}$. Then, the \emph{restriction} 
\begin{equation}
({\sf Restr}_s^t)_{\sharp}\pi\qquad\text{is a } q\text{-test plan}.\label{eq:pirestr}
\end{equation}
\end{itemize}
The proofs of these facts are based on a direct check of {\rm i)}-{\rm ii)} in Definition \ref{def:qplan}. In particular, for the restriction, it is necessary to make use of the change of variable
$\Keq({\sf Restr}_s^t(\gamma))= \Keq(\gamma) |t-s|^q$ for every $\gamma \in AC^q([0,1],\X)$. We are finally ready to present the definition of Sobolev class via duality with $q$-test plans. 
\begin{definition}[Sobolev class]\label{def:Sp}
Let $\Xdm$ be a metric measure space and $p \in (1,\infty)$. A Borel function $f$ belongs to $S^p(\X)$, provided there exists $G \in L^p(\X,\mm)$ with $G\ge 0$, called $p$\emph{-weak upper gradient} so that
\begin{equation}\label{def:Spineq}
\int| f(\gamma_1)-f(\gamma_0)| \, \d \pi\le \iint_0^1 G(\gamma_t)\msp \, \d t\d\pi, \qquad \forall \pi \ q\text{-test plan}.
\end{equation}
\end{definition}
Let us comment on the well-posedness of the definition. The assignment $(t,\gamma) \mapsto G(\gamma_t)\msp$ is Borel (see, e.g. \cite{GP19}) and the right hand side in \eqref{def:Spineq} is finite since the properties of any $q$-test plan $\pi$ ensure
\begin{equation}
\begin{aligned}
\iint_0^1 G(\gamma_t)\msp \, \d t\d\pi &\le \left( \iint_0^1 G^p(\gamma_t)\, \d t\d\pi \right)^{1/p}\left( \iint_0^1 \msp^q\, \d t\d\pi\right)^{1/q} \\
&\le \Comp(\pi) \|G\|_{L^p(\mm)}\rmKe_q^{1/q}(\pi) < +\infty.
\end{aligned}\label{eq:finiteRHS}
\end{equation}
The above calculation shows at the same time finiteness of the integral and  continuity of the assignment $L^p(\mm) \ni G \mapsto \iint_0^1 G(\gamma_t))\msp \, \d t\d\pi$. This, combined with the fact that convex combinations of $p$-weak upper gradients is a $p$-weak upper gradient, shows that the set of $p$-weak upper gradients of a given Borel function $f$, is a closed (because \eqref{eq:finiteRHS} shows that the map sending $G\in L^p(\mm)$ to the RHS of \eqref{def:Spineq} is continuous for any $q$-test plan $\pi$) and convex subset of the uniformly convex Banach space $L^p(\mm)$. The minimal $p$-weak upper gradient, denoted by $|Df|_p$, is then the element of minimal norm in this class. Also, by making use of the lattice property of the set of $p$-weak upper gradients (see \cite[Proposition 2.17 and Theorem 2.18]{Cheeger00}), such minimality is also in the $\mm$-a.e. sense. 

It will be useful to keep in mind the locality property of minimal weak upper gradients (see e.g.\ \cite[Proposition 4.8-(b)]{AmbrosioGigliSavare11} for the case $p=2$ and \cite[Proposition 5.2]{AmbrosioGigliSavare11-3} for the general case) i.e.:
\[ 
|Df|_p=|Dg|_p \quad \mm\text{-a.e. on } \{f=g\},\qquad \forall f,g \in S^{p}(\X),p \in (1,\infty),
\] 
and the Leibniz rule i.e.: for $f,g\in S^{p}(\X)\cap L^\infty (\mm)$ we have $fg\in S^{p}(\X)\cap L^\infty(\mm)$ with
\begin{equation}
\label{eq:leibwug}
|D(fg)|_p\leq |f||Dg|_p+|g||Df|_p,\qquad\mm\text{-a.e.}.
\end{equation}


Let us comment on  the relation between minimal $p$-weak upper gradients with respect to different $p$'s: for $p_1\le p_2$,  H\"older inequality easily implies $\{q_1 \text{-test plans} \} \subseteq \{q_2 \text{-test plans} \}$. Therefore the implication
\begin{equation}
p_1 \le p_2\text{ and } \begin{array}{l} f \in S^{p_2}(\X)\\ |Df|_{p_1} \in L^{p_1}(\mm)\end{array} \quad \Rightarrow \quad \begin{array}{l} f \in S^{p_1}(\X)\\ |Df|_{p_2}\text{ is a $p_1$-weak upper gradient} \\|Df|_{p_1} \le |Df|_{p_2} \quad \mm\text{-a.e.}

\end{array} \label{eq:Dfp1p2}
\end{equation} 
can be easily seen to be true on an arbitrary metric measure space. On the other hand, as pointed out already in the Introduction, there is no evidence that a function $f \in S^{p_1}(\X)\cap S^{p_2}(\X)$ satisfies $|Df|_{p_1} = |Df|_{p_2}$ $\mm$-a.e..\ Indeed, this fact is false in general and strict inequality may occur, see e.g. the analysis in \cite{DiMarinoSpeight13}. This motivates the choice of the $p$-exponent subscript for $|Df|_p$.

The prototype of functions belonging to the Sobolev class are the Lipschitz ones. We define the \emph{local Lipschitz constant} as the function
\[ \lip \ f(x) := \limsup_{y\rightarrow x} \frac{|f(y)-f(x)|}{\sfd(x,y)},\]
and $0$ if $x$ is isolated. A direct check of the definition shows then that $\Lip_{bs}(\X) \subseteq S^p(\X)$, and $\lip \ f$ is a $p$-weak upper gradient for every $ f \in \Lip_{bs}(\X)$ and $p \in (1,\infty)$.  In the sequel, we are going also to deal with Lusin-Lipschitz functions. We recall that a function $f \colon \X \rightarrow \R$ is Lusin-Lipschitz, provided there exists $N, K_n$ Borel, $n\in \N$, with $\X:= N\cup (\cup_n K_n)$, $N$ negligible and $K_n$ compacts, so that $f\restr{K_n}$ is Lipschitz for every $n \in \N$. In this case, it is convenient to adapt the notion of Lipschitz constant with the measure-theoretic notion of \emph{approximate local Lipschitz constant}
\[ \text{ap-lip }f(x) := \text{ap-}\limsup_{y\rightarrow x} \frac{|f(y)-f(x)|}{\sfd(x,y)},\]
and $0$ if $x$ is isolated, where the \emph{approximate limsup} of an arbitrary Borel function $u \colon \X\rightarrow \R$ is defined as
\[ \text{ap-}\limsup_{y\rightarrow x} u(x) := \inf \{t \in \R \colon x\text{ is a density point of }\{u\le t\}\},\]
recalling that, for a Borel set $E\subset \X$, $x \in E$ is a density point of $E$ if $\lim_{r\downarrow 0}\frac{\mm(B_r(x)\cap E)}{\mm(B_r(x)}=1$. On doubling spaces (see {\rm i)} in Definition \ref{def:Xdoubpoinc} below) it is standard to see, thanks to the existence of density points,  that the approximate Lipschitz constant  is \emph{local}, namely
\begin{equation}
\label{eq:locaplip}
\text{ap-lip }f = \text{ap-lip }g \qquad \mm\text{-a.e. on }\{f=g\}
\end{equation}
and it is not hard to check, from the fact that at a density point porosity cannot occur (see e.g.\  Proposition 2.5 in \cite{GT20}) that also
\begin{equation}
\label{eq:apliplip}
\text{ap-lip }f =\lip\, f \qquad \mm\text{-a.e. }\quad\forall f:\X\to\R\text{ Lipschitz}
\end{equation}
holds on doubling spaces.



Requiring a function in the Sobolev class to be also $p$-integrable leads to the definition of the full Sobolev space. We remark that the definition below makes sense, meaning that $W^{1,p}$ is a well-defined subset of $L^p$. In other words, the fact that a function $f$ belongs to a certain Sobolev class $S^p$ - and the corresponding notion of weak upper gradient - is invariant under modification of $f$ on a $\mm$-negligible set: this is due to property i) in Definition \ref{def:qplan} of test plans.
\begin{definition}[Sobolev space $W^{1,p}(\X)$]\label{def:W1p}
Let $\Xdm$ be a metric measure space and $p \in (1,\infty)$. The Sobolev space, denoted by $W^{1,p}(\X)$, is $L^p(\mm)\cap S^p(\X)$ as a set, equipped with the norm
\[ \|f\|_{W^{1,p}(\X)} := \big( \|f\|_{L^p(\mm)}^p +  \||Df|_{p}\|_{L^p(\mm)}^p\big)^{\frac1p}, \qquad \forall f \in W^{1,p}(\X).\]
\end{definition}
It is a standard fact that $W^{1,p}(\X)$ is a Banach space, by appealing to the weak lower semicontinuity of weak upper gradients (see \cite[Theorem 2.7]{Cheeger00} for the proof that such semicontinuity implies completeness). Also, it is in general false that it is reflexive and that the particular choice of $p=2$ leads  to a Hilbert space. When the latter situation occurs, we say that $\Xdm$ is \emph{infinitesimal Hilbertian} \cite{Gigli12}. Equivalently, we shall call a metric measure space infinitesimal Hilbertian provided the following \emph{parallelogram identity} is satisfied:
\begin{equation} 2|Df|_2^2 +2|Dg|_2^2 = |D(f+g)|_2^2+|D(f-g)|_2^2, \quad \mm\text{-a.e.}, \forall f,g \in W^{1,2}(\X). \label{eq:infhilb} \end{equation}

We finish this part by recalling the definition of \emph{local} Sobolev class.
\begin{definition}\label{def:Sploc}
Let $\Xdm$ be a metric measure space, $p \in (1,\infty)$. A real valued Borel function $f$ belongs to $S^p_{loc}(\X)$, provided there is $G\in L^p_{loc}(\mm)$ (i.e.\ the restriction of $G$ to bounded sets is in $L^p$), $G\geq 0$, such that for any $k>0$ and  $\eta\in \Lip_{bs}(\X)$, we have $\eta f^k \in S^{p}(\X)$ with
\[
|D (\eta f^k)|_p\leq |\eta| G,\qquad\mm\text{-a.e. on } \{ \eta =1\},	
\]
where $f^k:= k\wedge(f\vee- k)$. In this case, we define $|Df|_p \in L^p_{loc}(\mm)$ via
\[ |Df|_p = |D(\eta f^k)|_p \qquad \mm\text{-a.e. on } \{ \eta =1\}	\cap \{ |f|<k\},\]
for every $\eta$ and $k$ as before.
\end{definition}
We point out that the locality of the minimal $p$-weak upper gradient guarantees that the above definition is well-posed and the object $|Df|_p$ is $\mm$-a.e.\ well defined. It can also be proven that $f\in S^p_{loc}(\X)$ if and only if for some $G\in L^p_{loc}(\mm)$ non-negative \eqref{def:Spineq} holds.

We remark that in the generality we are working, the need of truncating the function before multiplying it by the cut-off is due to the following issue: there might be  $f\in S^p_{loc}(\X)$ and some $\eta\in  \Lip_{bs}(\X)$ such that $\eta f$ is not in $S^p(\X)$. Intuitively, this is due to the fact that the `best upper bound' for $|D (\eta f)|_p$ is $|f||D \eta|_p+|\eta||Df|_p$ and without any control on the size of $f$ it might be that this latter function is not in $L^p(\mm)$ (notice that even in the Euclidean setting there are a function $f\in L^1_{loc}$ with distributional differential in $L^p$ and a smooth compactly supported function $\eta$ such that the distributional differential of $\eta f$ is not in $L^p$).

\subsection{Optimal Transport on geodesic spaces}
We recall here basic features of Optimal transportation on a complete and separable metric space $(\X,\sfd)$. For a thorough discussion on the topic, we refer to the monograph \cite{Villani09} (see also \cite{AmbrosioGigli11}).

Let $q \in (1,\infty)$ and denote by $\cP_q(\X)$ the set of probabilities $\mu \in \cP(\X) $ with finite $q$-moment, i.e. $\int_\X \sfd^q(x,x_0)\, \d \mu <\infty$ for some (and thus, any) $x_0 \in \X$. Also, for a (possibly countable) cartesian product of the space $\X$, denote by $P^i$ and $P^{1,...,i}$ the canonical projections onto the $i$-th factor and the first $i$ factors, respectively. We equip $\cP_q(\X)$ with the \emph{Wasserstein distance}
\begin{equation}
 W_q(\mu_0,\mu_1) := \Big( \inf_{\alpha \in \Pi(\mu_0,\mu_1)}\int_{\X\times\X} \sfd^q(x,y)\, \d \alpha(x,y)\Big)^{1/q},\label{eq:Wp}
\end{equation}
where $\mu_0,\mu_1 \in \cP_q(\X)$ and $\Pi(\mu_0,\mu_1):= \{\alpha \in \cP(\X\times \X) \colon P^1_{\sharp}\alpha = \mu_0, P^2_{\sharp}\alpha= \mu_1\}$ is the set of \emph{admissible plans}. Any minimizer of \eqref{eq:Wp} is called \emph{optimal plan} and we denote $\Opt_p(\mu_0,\mu_1)$ the collection of optimal plans. We refer to $(\cP_q(\X),W_q)$ as the Wasserstein space. We report now some basic fact about convergence and compactness on the \emph{Wasserstein space}. Let $(\mu_n)\subseteq \cP_q(\X)$ and recall that
\begin{equation}
 W_q(\mu_n,\mu)\to 0 \quad \Leftrightarrow \quad \begin{array}{l}
\mu_n \weakto \mu \\
\int \sfd^q(x,x_0) \, \d \mu_n \rightarrow \int \sfd^q(x,x_0)\, \d \mu
\end{array} \label{eq:Wqweak}
\end{equation} 
as $n$ goes to infinity for $\mu \in \cP_q(\X)$ and for some (hence, any) $x_0 \in \X$. Moreover, $(\cP_q(\X),W_q)$ is complete and separable (if and only if $\X$ is so) and a family $\mathcal{K}$ is compact with respect to the topology induced by $W_q$ if and only if is tight and $q$-uniformly integrable.

 Throughout this manuscript, geodesics are always considered as constant speed curves defined on $[0,1]$ that realize the distance between their endpoints (and in particular are length  minimizing). The space $\Geo(\X)\subset C([0,1],\X)$ is the collection of such geodesics. Then, on a Polish geodesic spaces we can consider the set of \emph{dynamical optimal plans} between $\mu_0$ and $\mu_1$, defined as
\[ \OptGeo_q(\mu_0,\mu_1):= \{ \pi \in \cP(\Geo(\X)) \colon (\e_0,\e_1)_{\sharp}\pi\in \Opt_p(\mu_0,\mu_1) \}.\]
In this case, we have thus $(\e_0,\e_1)_\sharp \pi \in \Opt_q(\mu_0,\mu_1)$ and the kinetic energy realizes the Wasserstein distance; in fact it is not hard to see that we have
 \begin{equation}
W^q_q(\mu_0,\mu_1) =\rmKe_q(\pi), \quad \text{ if and only if }\quad  \pi \in \OptGeo_q(\mu_0,\mu_1). \label{eq:costgeo}
\end{equation}

In other words, the transportation cost can be equivalently evaluated as the superposition of kinetic energies of transportation geodesic in the support of optimal dynamical plans. Notice that existence of plans in $\OptGeo_q(\mu_0,\mu_1)$ implies the existence of `many' geodesics starting from the support of $\mu_0$ and ending on the support of $\mu_1$, but it can be that  $\OptGeo_q(\mu_0,\mu_1)$ is not empty for a large class of measures $\mu_0,\mu_1$ even if the underlying space $\X$ is not assumed to be geodesic. This is the case, for instance, of ${\sf CD}/{\sf RCD}(K,\infty)$ spaces and of spaces possessing the Bounded Interpolation Property that we introduce here (in all these examples it is not hard to see that geodesics with any given endpoints exist as soon as we assume bounded closed sets to be compact, but we shall not do so).

We shall sometimes consider the Wasserstein space also over the complete and separable space $(C([0,1],\X),\sfd_{\text{sup}})$. To avoid confusion, we will write in this situation $(\cP_q(C([0,1],\X),\mathcal{W}_q)$.


\subsection{Differential calculus}
Let $L^0(\mm)$ be the space of equivalence classes up to $\mm$-a.e. equality of Borel functions on $\X$ equipped with the topology of local convergence in measure. We recall the algebraic notion of a normed module over $\Xdm$ and discuss a nonsmooth differential calculus on metric measure spaces as in \cite{Gigli14}. In the following definition, we denote by $\hat{1}$ the equivalence class up to $\mm$-negligible set of the function constantly equal to one.
\begin{definition}[$L^0(\mm)$-normed module]
Let $\Xdm$ be a metric measure space. We call a $L^0(\mm)$\emph{-normed module} the quadruple $(\MM,\tau,\cdot,|\cdot|)$, where
\begin{itemize}
\item[{\rm i)}] $(\MM,\tau)$ is a topological space;
\item[{\rm ii)}] $\cdot \colon L^0(\mm) \times\MM \rightarrow\MM$ is a bilinear map satisfying the product's axioms \[ g\cdot(f\cdot v) = (fg)\cdot v, \quad \hat{{1}}\cdot v = v, \quad \forall f,g \in L^0(\mm), v \in \MM;\]
\item[{\rm iii)}] The map $|\cdot| \colon \MM \rightarrow L^0(\mm)$, called \emph{pointwise norm}, satisfying $|v|\ge 0$ and $|f\cdot v| = |f||v|$ $\mm$-a.e. for every $f\in L^0(\mm),v\in\MM$, is s.t. the function $\sfd_{\MM} \colon \MM\times \MM \rightarrow [0,\infty]$ defined via
\[ \sfd_{\MM}(v,w) :=\int_{\X} |v-w| \wedge 1 \, \d \mm', \quad \text{for some chosen, fixed } \mm' \text{ so that } \mm \ll \mm'\ll \mm,\]
is a complete distance on $\MM$ inducing the topology $\tau$.
\end{itemize}
\end{definition}
One can check that point {\rm iii)} does not make such definition ill posed, since the metric $\sfd_{\MM}$ may depend on the choice of $\mm'$, but the induced topology $\tau$ on $\MM$ does not. We recall the existence and uniqueness theorem of a suitable cotangent structure over a metric measure space in the language of $L^0(\mm)$-normed modules (see \cite[Theorem 1.8]{Gigli17} for the proof in the case $p=2$ in the language of $L^2$-normed modules and \cite[Theorem 2.4]{GR17} for that in the language of $L^0$-normed modules, and still the case $p=2$ - notice that these arguments adapt to general $p\in(1,\infty)$ without relevant modifications).
\begin{theorem}[$p$-Cotangent module] \label{thm:ctgmodule}
Let $\Xdm$ be a metric measure space and $p \in (1,\infty)$. Then, there is a unique couple $(\Lop,\d_p)$ where $\Lop$ is a $L^0(\mm)$-normed module and $\d_p \colon S^{p}_{loc}(\X) \rightarrow L^0_p(T^*\X)$ is linear and satisfying
\begin{itemize}
\item[{\rm i)}] For any $f \in S^{p}_{loc}(\X)$, it holds $|\d_pf|=|Df|_p$ $\mm$-a.e.;
\item[{\rm ii)}] The space $\{\d_pf \colon f \in W^{1,p}(\X)\}$ generates $\Lop$.
\end{itemize}
Here, uniqueness is intended up to unique module isomorphism, i.e. if $(\MM,L)$ is another couple with the same properties, then there is a unique isomorphism $\Phi \colon \MM\rightarrow \Lop$ so that $\Phi\circ L=\d_p$.
\end{theorem}
In the above statement, by generating, we mean that simple $L^0$-linear combinations are $\sfd_{L^0_p(T^*\X)}$-dense and by module isomorphism, we mean a map $\Phi$ preserving the module's operation that is also a pointwise isometry. Moreover, we sometimes informally call any element of the $p$-Cotangent module `Borel covector field'. Motivated by the need to discuss in this note also tangent structures over a metric measure space it is natural to give a definition of \emph{dual} of a module.
\begin{definition}[Dual of $L^0(\mm)$-normed module]
Let $\Xdm$ be a metric measure space and $\MM$ be a $L^0(\mm)$-normed module. Then, we define its \emph{dual module} $\MM^*$ as 
\[\MM^*:= \{ L \colon \MM \rightarrow L^0(\mm) \colon L\text{ is }L^0(\mm)\text{-linear and continuous} \}, \]
equipped with the following operations
\[\begin{split} 
(L+L')(v)&:=L(v)+L'(v), \\
(f\cdot L)(v)&:= L(f\cdot v), \\
|L|_*&:= \emph{ess sup }\{L(v) \colon v \in \MM, |v|\le 1\ \mm\text{-a.e.}\},
\end{split}\]
for any $f \in L^0(\mm)$, $L,L' \in \MM^*, v \in \MM$. 
\end{definition}
It is an easy task to check that $\MM^*$ has a natural $L^0(\mm)$-normed module structure (in particular, $|L|_*<\infty$ $\mm$-a.e. - this can be proved e.g.\ by contradiction assuming that $|L|_*=+\infty$ $\mm$-a.e.\ on a Borel set $A$ with positive measure. Indeed in this case with a cut-off procedure we can find $v_n$ with $|v_n|\leq 1$  and $|L(v_n)|\geq n^2$ on $A$. Then $\nchi_Av_n/n\to 0$ but $|L(\nchi_Av_n/n)|\geq1$ on $A$ for every $n\in\N$, contradicting continuity). Moreover, we can then define the $p$-\emph{tangent module} simply as
\[ 
L^0_p(T\X) := \big(L^0_p(T^*\X)\big)^*,\qquad \forall p \in (1,\infty). 
\]
An $L^0(\mm)$-normed module is called a Hilbert module provided
\[
2|v|^2+2|w|^2=|v-w|^2+|v+w|^2,\qquad\mm\text{-a.e.},
\]
holds for any $v,w\in \MM$. It is easy to check that a module is Hilbert if and only if its dual is so.  On  infinitesimally Hilbertian metric measure spaces, from \eqref{eq:infhilb} it follows that $L^0(T^*\X)$, and thus also $L^0(T\X)$ are Hilbert modules.

In the sequel, we shall sometimes require to work with $p$-integrable covector and vector fields among the Borel ones. In these situations, we restrict the attention to the spaces 
\[ \begin{aligned}
	&L^p(T^*\X) := \{ \omega \in L_p^0(T^*\X)  \colon |v| \in L^p(\mm) \},    &&\| \omega \|^p_{L^p(T^*\X)}:=  \int|\omega|^p\, \d \mm, \\
	&L^q(T\X) := \{ X \in L^0_q(T\X) \colon |X|_* \in L^q(\mm)\},     &&\| X \|^q_{L^q(T\X)} :=  \int | X|_*^q\, \d \mm,
\end{aligned}\]
which have a natural structure of $L^p(\mm)$-normed (resp. $L^q(\mm)$-normed) $L^\infty(\mm)$-module. We will not discuss such structure and refer to the monograph \cite{Gigli14} for a detailed discussion. Here, we shall only recall that they are modules over the commutative ring $L^\infty(\mm)$ and that they are Banach spaces.

\section{Independence under interpolation density bounds}

\subsection{Weak and strong $p$-independent gradients}
As already highlighted in \eqref{eq:Dfp1p2}, we shall only expect one inequality between minimal weak upper gradients with different $p$'s.  In light of the two different pathological situations described in the Introduction, we give the following definition.
\begin{definition}[$p$-independent weak upper gradients]\label{def:pindependent}
Let $\Xdm$ be a metric measure space. We say that it has $p$-independent weak upper gradients in the weak sense, provided for any $p_1,p_2 \in (1,\infty)$:
\begin{itemize}
\item[a)] $W^{1,p_1}(\X)\cap W^{1,p_2}(\X)$ is dense in both $W^{1,p_1}(\X)$ and $W^{1,p_2}(\X)$;
\item[b)] for any $f 	\in  W^{1,p_1}(\X)\cap W^{1,p_2}(\X)$ it holds $|Df|_{p_1}=|Df|_{p_2}$ $\mm$-a.e.;
\end{itemize}
Moreover, we say that $\X$ has $p$-independent weak upper gradients in the strong sense if we require \emph{a)-b)} and
\begin{itemize}
\item[c)] any $f\in W^{1,p_1}(\X)$ with $f,|Df|_{p_1} \in L^{p_2}(\mm)$ belongs to $W^{1,p_2}(\X)$.
\end{itemize}
\end{definition}
A non-trivial fact about doubling spaces supporting a Poincar\'e inequality is that they have $p$-independent weak upper gradients in the strong sense.  This is a consequence of the fact that on these spaces, in \cite{Cheeger00} it has been proved that the equality 
\begin{equation}
\label{eq:dach}
\lip(f)=|Df|_p,\qquad\mm\text{-a.e.},
\end{equation}
holds for any $ f\in \Lip_{bs}(\X)$ and of an approximation argument based on partition of the units done in \cite{KKST12} that shows that c) holds as well (see also \cite{AmbrosioGigliSavare11-3} for the proof that the different definition of Sobolev spaces in these references agree). Let us give some details, starting from the definitions.
\begin{definition}\label{def:Xdoubpoinc}
Let $\Xdm$ be a metric measure space. We say that
\begin{itemize}
\item[{\rm i)}] it is \emph{uniformly locally doubling} provided, for every $R>0$, there exists a constant ${\sf C}:={\sf C}(R)$ so that
\[  \mm(B_{2r}(x)) \le \sfC \mm(B_r(x)), \qquad \forall x \in \X,r\in (0,R).\]
For brevity, we shall only say that $\Xdm$ is a doubling metric measure space.
\item [{\rm ii)}] it supports a \emph{weak local $(1,1)$-Poincar\'e inequality}, provided for every $R>0$ there exists $\tau,\Lambda>0 $ so that for any $f \colon \X\rightarrow \R$ Lipschitz it holds
\[ \fint_{B_r(x)}|f-f_{B_r(x)}|\, \d \mm \le \tau r\fint_{B_{\Lambda r}(x)} \lip \ f \, \d \mm, \qquad \forall r \in (0,R),x \in \X,\] 
with the convention $f_B:= \fint_B f\, \d \mm$ for every $B \in \cB(\X)$.
\end{itemize}
\end{definition}
The notion of Poincar\'e inequality is often given on metric measure space with the concept of \emph{upper gradient}, rather than local lipschitz constant. Nevertheless, by appealing to \cite{AmbrosioGigliSavare11-3}, the two approaches are fully equivalent (we remark that in the earlier paper \cite{Keith03} the equivalence between the two versions of Poincar\'e inequality was proved on doubling spaces, which is sufficient for the present purposes - the arguments in \cite{AmbrosioGigliSavare11-3} remain valid even without doubling). We then have the following well known result (see for instance \cite[Corollary A.9]{Bjorn-Bjorn11} for a proof of the `weak' part, i.e.\ points a),b) in Definition \ref{def:pindependent} and \cite{KKST12} for the `strong' part, i.e.\ point c) in Definition \ref{def:pindependent}).
\begin{theorem}\label{prop:Chdoubpoincare}
Let $\Xdm$ be a doubling metric measure space supporting a weak local $(1,1)$-Poincar\'e inequality. Then it has $p$-independent weak upper gradients in the strong sense.
\end{theorem}

\subsection{The Bounded interpolation property (BIP)}
Aim of this section is to present our first order regularity constraint over a metric measure space within its implication concerning Sobolev spaces. This will be done by imposing a special behavior of the spreading of mass along the geodesic of the space according to the next definition.
\begin{definition}[Bounded interpolation property]\label{def:bip}
We say that a complete and separable metric measure space $\Xdm$ has the \emph{bounded interpolation property}, provided:

 for every $q\in (1,\infty)$ there exists a \emph{profile function}  $\R^+ \ni D\mapsto C_q(D) \in [1,\infty)$ so that for every $\mu_0 ,\mu_1 \in \cP(\X)$ absolutely continuous with bounded densities   and \emph{diam}$($\emph{supp}$(\mu_0) \cup$\emph{supp}$(\mu_1))<D$, there exists $\pi \in \OptGeo_q(\mu_0,\mu_1)$ satisfying 
\begin{equation}
(\e_t)_{\sharp}\pi = \rho_t\mm, \quad  \|\rho_t\|_{L^{\infty}} \le C_q(D)\|\rho_0\|_{L^{\infty}(\mm)} \vee\|\rho_1\|_{L^{\infty}(\mm)}, \qquad \forall t \in [0,1] \tag{BIP} \label{BIP},
\end{equation}
where $\rho_i:=\frac{\d\mu_i}{\d\mm}$. 

When this holds, we say for brevity that $\Xdm$ is a BIP-space, or that it  has the (BIP) with profile function $D\mapsto C_q(D)$.
\end{definition}
This axiomatization is inspired by the results of \cite{Rajala12-2}, where the very same special behavior of mass transportation has been investigated under synthetic lower Ricci bounds. Such analysis was carried for the exponent $q=2$, but we will see in Appendix \ref{B} that it can actually be performed for all $q \in (1,\infty)$. In this direction, \cite[Theorem 4.1]{Rajala12-2} ensures that the (BIP) yields the following:
\[ \text{A BIP-space supports a weak local }(1,1)\text{-Poincar\'e inequality.}\]
Thus we already know from Theorem \ref{prop:Chdoubpoincare} that if a BIP space is also doubling, then it has $p$-independent weak upper gradients in the weak sense. Our goal in this section is to show that, regardless of the doubling assumption, a BIP space has $p$-independent weak upper gradients in the strong sense. We then postpone to Section \ref{Sec5} the study of which sort of `known' spaces satisfy  (BIP). 

Notice that, for every $q$, it is not restrictive to suppose the profile function $D\mapsto C_q(D)$ to be nondecreasing and continuous, thus we shall implicitly use these facts sometime. Moreover, when the profile function is independent of $q$, as it will be in all the cases faced in Section \ref{Sec5}, we shall omit the subscript and simply write $D\mapsto C(D)$. 

For the sake of conciseness, we collect, for every $q \in (1,\infty)$ all the relevant interpolants in the class
\begin{equation}
{\sf Geod}_q (\X):= \left\{ \pi \in \OptGeo_q(\rho_0\mm,\rho_1\mm) \colon \begin{array}{l}
      D >0, \rho_0,\rho_1 \in L^\infty(\mm) \text{ probabilities}\\ $diam$($supp$(\rho_0) \cup $supp$(\rho_1))\le D \\
       (\e_t)_{\sharp}\pi \le C_q(D) \|\rho_0\|_{L^{\infty}(\mm)}\vee \|\rho_1\|_{L^{\infty}(\mm)} \mm
 \end{array} \right\}, \label{eq:Geodq}
\end{equation}
Notice the important fact that, no matter of the fixed exponent $q$, the defining property of this class ensures that
\begin{equation}
\label{eq:geodtest}
\text{any $\pi \in {\sf Geod}_q(\X)$ is an $\infty$-test plan, and thus also a $q'$-test plan for any $q'\in(1,\infty)$.}
\end{equation}
Indeed, every plan is concentrated on geodesics of the space whose lengths are controlled from above by some diameter. We shall also work with the `polygonal' version ${\sf PolGeo}_q(\X)$ of the above, defined as the set of plans $\pi\in\cP(C([0,1],\X))$ for which there are a finite Borel partition $(A_i)_{i=1,\ldots,N}$ of $C([0,1],\X)$ with $\alpha_i:=\pi(A_i)>0$  and, for every $j=0,\ldots,m-1, m\in \N$ and $i=1,\ldots,N$, we have $\alpha_i^{-1}({\sf Restr}_{\frac jm}^{\frac{j+1}m})_\sharp(\pi\restr{A_i})\in {\sf Geod}_q(\X)$.

\begin{lemma}[Approximation with polygonal plans]\label{lem:approxPol}
Let $(\X,\sfd,\mm)$ be a BIP-space, $q\in(1,\infty)$ and $\pi$ a $q$-test plan. Then there are $(\pi_n)\subset  \cP(C([0,1],\X))$ and $(\pi_{n,m})\subset {\sf PolGeo}_q(\X)$, $n,m\in\N$, such that:
\begin{itemize}
\item[i)] for every $n\in\N$, we have
\begin{itemize}
\item[a)]  $\pi_{n,m}\weakto \pi_n$ as $m\to\infty$;
\item[b)] $\lims_{m\to\infty}\rmKe_{q}(\pi_{n,m})\leq \rmKe_{q}(\pi_n)$;
\item[c)] for some $C(q,n)>0$ we have $(\e_t)_\sharp\pi_{n,m}\leq C(q,n)\mm$ for every $m\in\N$, and $t\in[0,1]$;
\end{itemize} 
\item[ii)] and moreover
\begin{itemize}
\item[a)] $\pi_{n}\weakto \pi$ as $n\to\infty$;
\item[b)] $\lims_{n\to\infty}\rmKe_{q}(\pi_{n})\leq \rmKe_{q}(\pi)$;
\item[c)]  for some $C(q)>0$ we have $(\e_t)_\sharp\pi_{n}\leq C(q)\mm$ for every $n\in\N$, and $t\in[0,1]$.
\end{itemize}
\end{itemize}
\end{lemma}
\begin{proof} 
Let us set for brevity $\Y:= C([0,1],\X)$ and assume at first that $\pi$ has compact support so that $E:= \{\gamma_t\colon \gamma \in \text{supp}(\pi), t \in [0,1]\} \subset \X$ is also compact. We put  $D:=$ diam$(E)<\infty$.

\noindent\textsc{Case }i). Let $m,n\in\N$ be fixed. Using the compactness of the support of $\pi$, find a finite Borel partition $(A_i)_{i=1,\ldots,N_n}$ of its support made of sets with positive $\pi$-measure and diameter $\leq \frac1n$. For $i\in\{1,\ldots,N_n\}$ put $\alpha^i:=\pi(A_i)^{-1}\pi\restr{A_i}$ and then for $j\in\{0,\ldots,m-1\}$ let $\beta^{i,j}\in\OptGeo_q((\e_{\frac jm})_\sharp\alpha^i,(\e_{\frac {j+1}m})_\sharp\alpha^i)\cap {\sf Geod}_q(\X)$ given by the (BIP). With a gluing argument (see e.g.\ \cite[Lemma 2.1.1]{G11}) we can then find a plan $\bbeta^i$ such that $({\sf Restr}_{\frac jm}^{\frac{j+1}m})_\sharp\bbeta^{i}=\beta^{i,j}$ for every $j\in\{0,\ldots,m-1\}$. We put $\pi_{n,m}:=\sum_{i=1}^{N_n}\pi(A_i)\bbeta^i$ and we notice that the construction and the BIP assumption easily guarantees that property {\rm (i-c)} holds. Moreover, we claim
\begin{equation}
\label{eq:claimpinmke}
\rmKe_{q}(\pi_{n,m})\leq \rmKe_{q}(\pi)
\end{equation}
and to this aim we notice that
\[
\begin{split}
\iint_0^1 \msp^q\, \d t \d \bbeta^i &= \sum_{j=0}^{m-1}\iint_{\frac jm}^{\frac{j+1}m} |\dot{\gamma}_{t}|^q\, \d t\d \bbeta^i \\
\Big(\text{from }({\rm Restr}_{\frac jm}^{\frac{j+1}m})_\sharp \bbeta^i = \beta^{i,j}\Big)\ \qquad& = \sum_{j=0}^{m-1}\iint_0^1m^{q-1}\msp^q\, \d t \d \beta^{i,j} \\
\text{(by \eqref{eq:costgeo})} \ \qquad&= \sum_{j=0}^{m-1}m^{q-1}W_q^q\big((\e_{\frac jm})_\sharp\alpha^i,(\e_{\frac{j+1}m})_\sharp\alpha^i\big) \\
&\le \sum_{j=0}^{m-1}m^{q-1}\int \sfd^q(\gamma_{\frac jm},\gamma_{\frac{j+1}m})\, \d \alpha^i\\
 &\le \sum_{j=0}^{m-1}m^{q-1} \int\big(\int_{\frac jm}^{\frac{j+1}m}\msp\, \d t\big)^q\, \d \alpha^i \\
 \text{(by Jensen)}\ \qquad&\le \sum_{j=0}^{m-1}m^{q-1-\frac{q}{p}} \iint_{\frac jm}^{\frac{j+1}m}\msp^q\, \d t\d\alpha^i   = \iint_0^1 \msp^q\, \d t \d \alpha^i,
 \end{split}
\]
for all $i=0,...,N_n$. This in particular guarantees that $(\pi_{n,m})_m$ is a sequence of $q$-test plans with uniformly bounded $q$-kinetic energy and compression. We are going now to produce a weak limit $\pi_n$, arguing by tightness. 

Fix $t$, let $j:=j(t,m)$ so that $t \in [j/m,(j+1)/m]$ and, using that $\pi_{n,m} \in {\sf PolGeo}_q(\X)$, we estimate 
\begin{equation}
\label{eq:iic}
\begin{split}
W_q( (\e_t)_{\sharp}\pi_{n,m}, (\e_t)_{\sharp}\pi) &\le W_q( (\e_t)_{\sharp}\pi_{n,m},(\e_{\frac jm})_{\sharp}\pi_{n,m} ) + W_q((\e_{\frac jm})_{\sharp}\pi,(\e_t)_{\sharp}\pi)\\
&\le \left( \int \sfd^q(\gamma_{\frac jm},\gamma_t)\, \d \pi_{n,m} \right)^{1/q} +  \left( \int \sfd^q(\gamma_{\frac jm},\gamma_t)\, \d \pi \right)^{1/q} \\
(AC^q\text{-supported})\qquad &\le  \Big( \int \big(\int_{\frac jm}^t \msp\, \d t\big)^q \d \pi_{n,m} \Big)^{1/q} +  \Big( \int \big(\int_{\frac jm}^t \msp\, \d t\big)^q \d \pi \Big)^{1/q}    \\
(\text{H}\ddot{\text{o}}\text{lder and }\eqref{eq:claimpinmke})\qquad &\le  2m^{\frac{1}{pq}}\rmKe^{1/q} _q(\pi).
\end{split}
\end{equation}
Taking into account that $\{ (\e_t)_{\sharp}\pi:t\in[0,1]\}$ is $W_q$-compact (because $\pi$ has finite $q$-energy and thus $t\mapsto  (\e_t)_{\sharp}\pi\in (\cP_q(\X),W_q)$ is continuous), 
this last estimate ensures that $\{ (\e_t)_{\sharp}\pi_{n,m}:t\in[0,1],m\in\N\}$ is $W_q$-precompact for every $n\in\N$. 
In particular such set is tight, and thus by Prokhorov's Theorem \ref{thm:prokh} there exists a function $\psi : \X\rightarrow \R$ with compact sublevels such that
\begin{equation}
\label{eq:datight}
 \sup_{t \in[0,1],m\in\N} \, \int \psi \, \d(\e_t)_{\sharp}\pi_{n,m}<\infty.
\end{equation}
Now, consider the  functional $\Psi : C([0,1],\X) \rightarrow \X$ defined by
\[ \Psi(\gamma) := \int_0^1 \psi(\gamma_t)+ \msp^q\, \d t, \quad \text{if } \gamma \in AC^q([0,1],\X), \quad +\infty \text{ otherwise}\]
and notice that it has compact sublevels as well (see, e.g., \cite[Lemma 5.8]{GMS15} for the case $q=2$ and observe that for $q \in (1,\infty)$ the proof works as well). By construction we have
\[ 
\sup_m\int \Psi\, \d \pi_{n,m} \le\sup_{t \in[0,1],m\in\N} \Big(\, \int \psi \, \d(\e_t)_{\sharp}\pi_{n,m}  + \rmKe_q(\pi_{n,m})\Big) \stackrel{\eqref{eq:claimpinmke},\eqref{eq:datight}}<\infty.
\]
Again, by Prokhrov's Theorem, we conclude that $(\pi_{n,m})_m$ is tight family and, up to not relabeled subsequences, we get the existence of a weak limit $\pi_n$ as $m$ goes to infinity. We thus obtained  {\rm (i-a)}. Also, from \eqref{eq:iic} we get
\begin{equation}
(\e_t)_{\sharp}\pi_n = (\e_t)_{\sharp}\pi, \qquad t \in [0,1].\label{eq:etpi}
\end{equation}
Now notice  that \eqref{eq:etpi} ensures \emph{a posteriori} $(\pi_{n,m})_m$ to be also a \emph{polygonal interpolation} of $\pi_n$ (recall that $\pi_{n,m}$ was built freezing marginals of $(\e_t)_\sharp\pi$ on a uniform time grid) whence \eqref{eq:claimpinmke} here reads $\rmKe_q(\pi_{n,m})\le \rmKe_q(\pi_n)$ for every $m\in \N$. Taking now the limsup yields {\rm (i-b)}.

\noindent\textsc{Case }ii). We immediately notice that \eqref{eq:etpi} ensures {\rm (ii-c)} with $C=\Comp(\pi)$. Next, we show {\rm (ii-a)} and, to this aim, we remark that $\Delta^n := \sum_{i=0}^{N_n} \alpha_i\pi_n\restr{A_i}\otimes \pi\restr{A_i}  \in \Adm(\pi_n,\pi)$ by construction. Then, we can estimate
\[ \mathcal{W}^q_q(\pi^n,\pi) \le \int_{\Y\times \Y} \sfd^q_\text{sup}(\gamma,\theta) \, \d \Delta^n(\gamma,\theta) = \sum_{i=0}^{N_n}  \int_{A_i\times A_i} \alpha_i\sfd^q_\text{sup}(\gamma,\theta)\, \d \pi_n(\gamma) \pi(\theta) \le \tfrac{1}{n^q},\]
where, evidently, we used that for $\pi_n \otimes\pi $-a.e. $(\gamma,\theta) \in A_i\times A_i$ we have $\sfd_\text{sup}(\gamma,\theta) \le \frac 1n$ due to the uniform bound of the diameter of $A_i$. This clearly implies {\rm (ii-a)}. But now, arguing again by weak lower semicontinuity \eqref{eq:KEqlsc}, we conclude recalling {\rm (i-b)} and  \eqref{eq:claimpinmke} that
\[  \limsup_{n\to \infty} \rmKe_q(\pi_n) \le \limsup_{n\to \infty}\liminf_{m\to \infty} \rmKe_q(\pi_{n,m}) \le \rmKe_q(\pi), \]
that is {\rm (ii-b)}.

\noindent\textsc{Reduction step}. In this final step, we relieve the proof of the Lemma of assumption $\pi$ supported on a compact set. Being $\pi$ a probability measure on the complete and separable space $\Y$, it is concentrated on a sigma-compact set. Let then $\Gamma_k \subset \Y$ be compact so that $\pi(\cup_k\Gamma_k) =1$, and consider, for every $k \in \N$, the plans $\pi^k := \pi(\cup_{i\le k}\Gamma_i)^{-1}\pi\restr{\cup_{i\le k}\Gamma_i}$. They are clearly of compact support, so that we can apply {\rm i)}-{\rm ii)} to produce the sequences $\pi^k_{n,m}$ and $\pi_n^k$ satisfying all the listed properties. Now, a diagonalization argument in $k$ and $n$ gives the conclusion.

\end{proof}

\begin{lemma}\label{lem:limPol}
Let $(\X,\sfd,\mm)$ be a metric measure space, $q\in(1,\infty)$ and $(\pi_n)\subset\cP(C([0,1],\X))$ be a sequence such that $\pi_n\weakto \pi$ as $n$ goes to infinity for some $q$-test plan $\pi\in \cP(C([0,1],\X))$. Assume that
\begin{equation}
\label{eq:unifcomp}
(\e_t)_\sharp\pi_n
\leq C\mm,\qquad\forall n\in\N,\ t\in[0,1],
\end{equation}
and that
\begin{equation}
\label{eq:limskeq}
\lims_{n\to\infty}\rmKe_{q}(\pi_{n})\leq \rmKe_{q}(\pi).
\end{equation}
Then for every $G\in L^p(\mm)$ we have
\begin{equation}
\lim_{n\to\infty}\iint_0^1 G(\gamma_t)|\dot\gamma_t|\,\d t\,\d\pi_n=\iint_0^1 G(\gamma_t)|\dot\gamma_t|\,\d t\,\d\pi.\label{eq:limG}
\end{equation}
\end{lemma}
\begin{proof} Let $\sfd' := \sfd_{\rm sup}\vee 1$ and $\cW_q$ be the $q$-Wasserstein distance induced by $\sfd'$. Thus  $\cW_q(\pi,\pi_n)\to 0$ as $n$ goes to infinity because $\sfd'$ is a bounded distance equivalent to the original one.

We write again for brevity $\Y:=C([0,1],\X)$ and consider, for every $n \in \N$, first the plans $\beta^n \in \Opt_q(\pi,\pi_n)$ and then, using repeatedly a gluing argument, the plan $\bbeta^n \in \Probabilities{\Y\times \Y^n} $ so that
\[ (P^{0,n})_{\sharp}\bbeta^n = \beta^n \quad \text{and} \quad  (P^{0,1,...,n-1})_{\sharp}\bbeta^n = \bbeta^{n-1}.\]
Kolmogorov's Theorem (see e.g.\ \cite[Section 7.7]{Bogachev07}) ensures the existence of $\bbeta \in \cP(\Y\times \Y^{\N})$ so that $(P^{0,1,...,n})_{\sharp}\bbeta = \bbeta^n$ for all $n \in \N$. Thanks to the assumptions, we can write
\[ 
0=\lim_{n\to \infty} \mathcal{W}^q_q(\pi,\pi_n)=\lim_n \int_{\Y\times\Y} (\sfd')^q(\gamma^0,\gamma^n)\, \d \beta^n(\gamma^0,\gamma^n) =\lim_n \int_{\Y\times\Y^{\N}}  (\sfd')^q(P^0(\ggamma),P^n(\ggamma))\, \d \bbeta (\ggamma).
\] 
Therefore, one gets that, up to a not relabeled subsequence, $P^n(\ggamma)\rightarrow P^0(\ggamma)$ uniformly for $\bbeta$-a.e. $\ggamma$. Now let $f_n(\ggamma,t):=|\dot\gamma_t^n|$ and $g_n(\ggamma):=\int_0^1f^q_n(\ggamma,t)\,\d t=\Keq(\gamma^n)$ and similarly $f,g$. Notice that \eqref{eq:limskeq} reads as $\lims\int g_n\,\d\bbeta\leq \int g\,\d\bbeta$ and the lower semicontinuity of the $q$-kinetic energy ensures that $\limi g_n(\ggamma)\geq g(\ggamma)$ for $\bbeta$-a.e.\ $\ggamma$. Hence the simple Lemma \ref{le:convl1} below ensures that $g_n\to g$ in $L^1(\bbeta)$ and thus, up to a non-relabeled subsequence,  also $\bbeta$-a.e.. Thus by Lemma \ref{lem:closuremsp}  we deduce that for $\bbeta$-a.e.\ $\ggamma$ we have $f_n(\ggamma,\cdot)\to f(\ggamma,\cdot)$ in $L^q(0,1)$ and thus also in measure. By Fubini's theorem we then see that $f_n\to f$ in measure (w.r.t.\ $\bbeta\times\mathcal L^1\restr{[0,1]}$). Now observe that \eqref{eq:limskeq} (and the identity $\|f_n\|_{L^q}=\Keq(\pi_n)$) guarantees that $(f_n)$ is bounded in $L^q(\bbeta\times \mathcal L^1\restr{[0,1]})$ and what we just proved shows that any weak limit must coincide a.e.\ with $f$, i.e.\ $f_n\weakto f$ in $L^q(\bbeta\times \mathcal L^1\restr{[0,1]})$. Using again \eqref{eq:limskeq} and the uniform convexity of $L^q$ we conclude that 
\begin{equation}
\label{eq:fnlq}
f_n\to f\quad\text{  in }L^q(\bbeta\times \mathcal L^1\restr{[0,1]}).
\end{equation}
Putting $\hat G_n(\ggamma,t):=G(\gamma^n_t)$ and analogously $\hat G(\ggamma,t):=G(\gamma_t)$,  we then see that to conclude it is sufficient to show that $\hat G_n\to \hat G$ in $L^p(\bbeta\times \mathcal L^1\restr{[0,1]})$. This is obvious by dominated convergence if $G\in C_b(\X)$, thus the conclusion will follow if we show that the linear maps $L^p(\X,\mm)\ni G\mapsto \hat G_n,\hat G\in L^p(\bbeta\times \mathcal L^1\restr{[0,1]})$ are uniformly continuous. This follows from \eqref{eq:unifcomp}, which give
\[
\iint_0^1|\hat G_n|^p\,\d t\,\d\bbeta=\int_0^1\int |G|^p(\cdot,t)\,\d\pi_n\,\d t=\int_0^1\int |G|^p\,\d (\e_t)_\sharp\pi_n\,\d t\leq C\int|G|^p\,\d\mm,
\]
and the analogous estimates for $\pi$. Since the result does not depend on the particular subsequence chosen, the conclusion follows.
\end{proof}
\begin{lemma}\label{le:convl1}
Let $\mu$ be a Borel probability measure on a Polish space $\Y$, and $f_n,f:\Y\to[0,\infty]$, $n\in\N$, Borel such that 
\[
f(y)\leq \limi_{n\to\infty}f_n(y),\qquad\text{ and }\qquad \lims_{n\to\infty}\int f_n\,\d\mu\leq  \int f\,\d\mu<\infty.
\]
Then $f_n\to f$ in $L^1(\mu)$.
\end{lemma}
\begin{proof}
Let $g:= \limi_{n\to\infty}f_n$ and $g_n:=\inf_{k\geq n}f_k$. Fatou's lemma and the assumptions give that $\int g\,\d\mm<\infty$, while  the monotone convergence theorem ensures that $\int g_n\,\d\mm\to\int g\,\d\mm $. Hence  $\|g_n- g\|_{L^1(\mu)}\to 0$. Also, we have
\[
\lims_{n\to\infty}\|f_n-g_n\|_{L^1(\mu)}=\lims_{n\to\infty}\int f_n-g_n\,\d\mu\leq \int f-g\,\d\mu\leq 0,
\]
forcing in particular $f=g$ $\mu$-a.e.. The conclusion follows.
\end{proof}
Thanks to this approximation result we get the following:
\begin{proposition} \label{prop:sobolevPi}
Let $\Xdm$ be a BIP-space, $p \in (1,\infty)$, $f\colon \X \rightarrow \R$ Borel and $G \in L^p(\X)$ positive. Then, the following are equivalent:
\begin{itemize}
\item[{\rm i)}] $f \in S^p(\X)$ and $G$ is a $p$-weak upper gradient;
\item[{\rm ii)}] the inequality
\begin{equation}
\label{eq:persob}
\int |f(\gamma_1)-f(\gamma_0)|\, \d \pi \le \iint_0^1 G(\gamma_t)\msp \, \d t \d \pi
\end{equation}
holds for any $\pi \in {\sf Geod}_q(\X)$.
\end{itemize}
\end{proposition}
\begin{proof}
The implication {\rm i)} $\Rightarrow$ {\rm ii)} is obvious, so we are left to show the converse. We start  noticing that \eqref{eq:persob} holds for any $\pi\in  {\sf PolGeo}_q(\X)$. Indeed, for $A_i,\alpha_i,N,m$ as in the definition of  ${\sf PolGeo}_q(\X)$ we have
\[
\begin{split}
\int |f(\gamma_1)-f(\gamma_0)|\, \d \pi 
&\leq \sum_{i=1}^N\sum_{j=0}^{m-1}\int |f(\gamma_{\frac {j+1}m})-f(\gamma_{\frac jm})|\, \d \pi\restr{A_i} \\
&=\sum_{i=1}^N\sum_{j=0}^{m-1}\alpha_i\int |f(\gamma_1)-f(\gamma_0)|\, \d \big(\alpha_i^{-1}({\sf Restr}_{\frac jm}^{\frac{j+1}m})_\sharp(\pi\restr{A_i})\big)\\
&\stackrel{*}\leq\sum_{i=1}^N\sum_{j=0}^{m-1}\alpha_i \iint_0^1 G(\gamma_t)\msp \, \d t \d \big(\alpha_i^{-1}({\sf Restr}_{\frac jm}^{\frac{j+1}m})_\sharp(\pi\restr{A_i})\big)\\
&=\iint_0^1 G(\gamma_t)\msp \, \d t \d \pi
\end{split}
\]
having used the fact that $\alpha_i^{-1}({\sf Restr}_{\frac jm}^{\frac{j+1}m})_\sharp(\pi\restr{A_i})$ is in ${\sf Geod}_q(\X)$ and the assumption {\rm ii)} in the starred inequality.

The conclusion now comes by approximation. Let $\pi$ be an arbitrary $q$-test plan and assume for the moment that $f(\gamma_1)-f(\gamma_0)$ has the same sign for $\pi$-a.e.\ $\gamma$, say non-negative (otherwise replace $\pi$ with $({\sf Restr}_1^0)_\sharp\pi$ and notice that \eqref{eq:persob} is unaffected). Let $(\pi_{n,m}),(\pi_n)$ be given by Lemma  \ref{lem:approxPol}, put $f^k:=(-k)\vee f\wedge k$ for $k\in\N$ and notice that $f^k(\gamma_1)-f^k(\gamma_0)\geq 0$ for $\pi$-a.e.\ $\gamma$. The fact that $f^k\in L^\infty(\X)$ and the compression bounds given by Lemma \ref{lem:approxPol} give that
\[
\lim_{n\to\infty }\lim_{m\to\infty }\int f^k\,\d(\e_t)_\sharp\pi_{n,m}=\lim_{n\to\infty }\int f^k\,\d(\e_t)_\sharp\pi_{n}=\int f^k\,\d(\e_t)_\sharp\pi,\qquad\forall t\in[0,1],
\]
therefore by monotone convergence we get
\[
\begin{split}
\int |f(\gamma_1)-f(\gamma_0)|\,\d\pi&=\lim_{k\to\infty}\int |f^k(\gamma_1)-f^k(\gamma_0)|\,\d\pi\\
&=\lim_{k\to\infty}\int f^k(\gamma_1)-f^k(\gamma_0)\,\d\pi\\
&=\lim_{k\to\infty}\Big(\int f^k\,\d(\e_1)_\sharp\pi-\int f^k\,\d(\e_0)_\sharp\pi\Big)\\
&=\lim_{k\to\infty}\lim_{n\to\infty }\lim_{m\to\infty }\Big(\int f^k\,\d(\e_1)_\sharp\pi_{n,m}-\int f^k\,\d(\e_0)_\sharp\pi_{n,m}\Big)\\
&\leq\limi_{k\to\infty}\limi_{n\to\infty }\limi_{m\to\infty }\int |f^k(\gamma_1)-f^k(\gamma_0)|\,\d \pi_{n,m}\\
&\leq\limi_{n\to\infty }\limi_{m\to\infty }\int |f(\gamma_1)-f(\gamma_0)|\,\d \pi_{n,m}\\
&\leq\limi_{n\to\infty }\limi_{m\to\infty }\iint_0^1 G(\gamma_t)\msp \, \d t \d \pi_{n,m},
\end{split}
\]
where in the last step we used the fact that $\pi_{n,m}\in {\sf PolGeo}_q(\X)$ and what previously proved. To conclude we apply Lemma \ref{lem:limPol} first as $m\to\infty$ and then as $n\to\infty$. 

If both $A^+:=\{\gamma: f(\gamma_1)-f(\gamma_0)\geq 0\}$ and $A^-:=\{\gamma: f(\gamma_1)-f(\gamma_0)< 0\}$ have positive $\pi$-measure, we apply the above to the $q$-test plans $\pi^\pm:=\pi(A^\pm)^{-1}\pi\restr{A^\pm}$ observing that
\[
\int |f(\gamma_1)-f(\gamma_0)|\,\d\pi=\pi(A^+) \int |f(\gamma_1)-f(\gamma_0)|\,\d\pi^+-\pi(A^-)\int| f(\gamma_1)-f(\gamma_0)|\,\d\pi^-.
\]
The conclusion follows.
\end{proof}
We now come to the main result of the section, which is also the main reason behind the definition of BIP spaces:

\begin{theorem}\label{thm:pind}
Let $\Xdm$ be a BIP-space and let $p_1,p_2 \in (1,\infty)$. Suppose $f \in S^{p_1}_{loc}(\X)$ is such that $|Df|_{p_1} \in L^{p_2}_{loc}(\mm)$. 

Then, $f \in S^{p_2}_{loc}(\X)$ and 
\[ |Df|_{p_1}=|Df|_{p_2}, \qquad \mm\text{-a.e.}.\]
\end{theorem}
Note: by a truncation and cut-off argument it is easy to see that the conclusion of this theorem is in fact equivalent to the $p$-independence of weak upper gradients in the strong sense.
\begin{proof} Assume for a moment that $f \in S^{p_1}(\X)$ and $|Df|_{p_1} \in L^{p_2}(\mm)$. Then we know that \eqref{eq:persob} holds for every $q_1$-test plan with $G:=|Df|_{p_1}$, and thus, recalling \eqref{eq:geodtest}, also for every plan in ${\sf Geod}_{q_2}(\X)$. Hence Proposition \ref{prop:sobolevPi} tells that $f \in S^{p_2}(\X)$  with $|Df|_{p_2}\leq |Df|_{p_1}$ $\mm$-a.e..

In the general case we pick $k\in\N$  and $\eta\in\Lip_{bs}(\X)$, define the truncated function $f^k:=(-k)\vee f\wedge k$ and then consider $\eta f^k$. The Leibniz rule \eqref{eq:leibwug} (which is trivially valid also for locally Sobolev functions) gives 
\begin{equation}
\label{eq:p1p2}
|D (\eta f^k)|_{p_1}\leq |\eta||D f|_{p_1}+|D\eta|k\quad\in L^{p_1}\cap L^{p_2}(\mm).
\end{equation}
Thus $\eta f^k\in S^{p_1}(\X)$ with $|D (\eta f^k)|_{p_1}\in L^{p_2}(\mm)$ and the previous argument applies to conclude that $\eta f^k\in S^{p_2}(\X)$ with $|D (\eta f^k)|_{p_2}$ bounded by the right hand side of \eqref{eq:p1p2}. By the very Definition \ref{def:Sploc} this means that $f\in  S^{p_2}_{loc}(\X)$ with $|D f|_{p_2}\leq |D f|_{p_1}$ $\mm$-a.e..

Now we can swap $p_1$ and $p_2$ to get that equality, and thus the conclusion, holds.
\end{proof}


\subsection{Stability of (BIP)}
We aim at proving the stability of the (BIP) under pointed measured Gromov Hausdorff convergence. We recall here some basic facts about limits of metric measure spaces as introduced in \cite{Gromov07} (see also \cite{Sturm06I}, here we follow the extrinsic approach described in \cite{GMS15}). A pointed metric measure space, is a quadruple $\Xdmx$, where $\Xdm$ is a metric measure space and $x  \in \X$. Also, we put $\bar{\N}:=\N\cup\{\infty\}$.
\begin{definition}[pmGH-convergence]\label{def:pmGH}
Let $\Xdmxn$, $n \in \bar{\N}$, be a sequence of pointed metric measure spaces. We say that that $\Xdmxn$ \emph{pointed-measure Gromov Hausdorff-converges} (pmGH-converges for short) to $\Xdmxinf$ provided there exists a complete and separable metric measure space $(\Z,\sfd)$ and isometric embeddings
\[
\begin{split}
\iota_n &\colon (\X_n,\sfd_n) \to (\Z,\sfd_\Z), \\
\iota_\infty &\colon (\X_\infty,\sfd_\infty) \to (\Z,\sfd_\Z),
\end{split}
\]
such that $(\iota_n)(x_n) \to \iota_\infty(x_\infty)$ and
\[ (\iota_n)_\sharp \mm_n \weakto (\iota_\infty)_\sharp \mm_\infty, \qquad \text{in duality with }C_{bs}(\Z).\]
In this case we  write $\X_n \overset{pmGH}{\rightarrow} \X_\infty$.
\end{definition}
In what follows we shall identify the spaces $\X_n$, $n\in\bar\N$, with their isomorphic images in $\Z$.

The stability of the (BIP) is a consequence of the following simple compactness  property of optimal geodesic test plans under pmGH-convergence:
\begin{lemma}\label{lem:pipmGH} Let $\X_n \overset{pmGH}{\rightarrow} \X_\infty$ as in Definition \ref{def:pmGH},   $R>0$ and $q \in (1,\infty)$. For every $n \in \N$, and $i=0,1$, let $\mu_i^n \in \cP_q(\X_n)$ be with \emph{supp}$(\mu^n_i) \subseteq B_R(x_n)$ and $\pi^n \in \OptGeo_q(\mu_0^n,\mu_1^n)$. Assume that $\lims_{n\to\infty}\Comp(\pi^n) <\infty$.

Then $(\pi^n)\subset \cP(C([0,1],\Z))$ is tight and for any weak limit $\pi$ along some subsequence $n_k\uparrow+\infty$ we have 
\begin{itemize}
\item[i)] $\mu^{n_k}_i\weakto  \mu_i$, $i=0,1$, for some $\mu_i\in \cP(\X_\infty)\subset\cP(\Z)$ with support contained in $\bar B_R(x_\infty)$,
\item[ii)] $\pi\in \OptGeo_q(\mu_0,\mu_1)$ and $\lim_{n\to \infty}\rmKe_q(\pi^n) = \rmKe_q(\pi)$;
\item[iii)]  $\Comp(\pi)\leq \limi_{n\to\infty}\Comp(\pi^n)$.
\end{itemize}
\end{lemma}
\begin{proof} We start passing to a subsequence, not relabeled, realizing the $\limi_{n\to\infty}\Comp(\pi^n)$.  From \eqref{eq:costgeo} and the fact that the measures $\mu^n_i$ have uniformly bounded supports in $\cP(\Z)$ and arguing as in the proof of Lemma \ref{lem:approxPol} we see that tightness will follow if we show that $\{(\e_t)_\sharp\pi^n:t\in[0,1],n\in\N\}$ is tight. To see this, let $\eta:\Z\to[0,1]$ be Lipschitz with bounded support and identically 1 on $B_{R+1}(x_\infty)\subset\Z$. Since $x_n\to x_\infty$ we see that $\int\eta\,\d\mm_n>0$ for every $n$ sufficiently big, hence the  definition
\[ 
\tilde\mm_n :=\tfrac1{z_n}\eta  \mm_n \in \cP(\Z), \qquad z_n:= \int \eta \, \d\mm_n, 
\]
is well posed for every $n\in\N$ sufficiently big and we have  $\tilde\mm_n\weakto \tilde\mm_\infty $ in duality with $C_b(\Z)$. In particular, $(\tilde\mm_n)$ is a tight sequence. Now fix $\eps>0$ and  find $K\subset\Z$ compact such that $\lims_n\tilde\mm_n(\Z\setminus K)<\eps$. Then observe that for every $n\in\N$ big enough we have ${\rm supp}((\e_t)_\sharp\pi^n)\subset\{\eta=1\}$ and thus 
\begin{equation}
\label{eq:1}
(\e_t)_\sharp\pi^n\leq\Comp(\pi^n)\tilde\mm_n,\qquad\text{ for every $t\in[0,1]$ and $n$ big enough.}
\end{equation}
Hence for any $S>\lim_n\Comp(\pi^n)$ we have 
\[
\lims_{n\to\infty}\sup_{t\in[0,1]}(\e_t)_\sharp\pi^n(\Z\setminus K)\leq \lim_{n\to\infty}\Comp(\pi^n)\tilde\mm_n(\Z\setminus K)\leq S\eps,
\]
proving the desired tightness.

Now say that  $\pi^{n_k}\weakto\pi\in\cP(C([0,1],\Z))$. Then $(\e_t)_\sharp\pi^n\weakto(\e_t)_\sharp\pi$ for every $t\in[0,1]$ (in particular {\rm i)} holds), thus passing to the limit in \eqref{eq:1} we obtain 
\[
(\e_t)_\sharp\pi\leq S\tilde\mm_\infty\leq S\mm_\infty,\qquad\forall t\in[0,1]\text{ and }S>\lim_n\Comp(\pi^n),
\]
thus {\rm iii)} holds. To see {\rm ii)} notice that since the measures $\mu^n_i$ have uniformly bounded support, the weak convergence $\mu^{n_k}_i\weakto (\e_i)_\sharp\pi$ implies $W_q$-convergence, thus recalling the characterization \eqref{eq:costgeo} of optimal geodesic plans we have
\[
\Keq(\pi)\leq\limi_{k\to\infty}\Keq(\pi^{n_k})=\limi_{k\to\infty}W_q^q(\mu^{n_k}_0,\mu^{n_k}_1)=W_q^q(\mu_0,\mu_1)
\]
and the conclusion follows.
\end{proof}
We come to the actual stability result:
\begin{theorem}[pmGH-stability of (BIP)]\label{thm:pmGHBIP}
Let $\Xdmxn$, $n \in \bar{\N}$, be a sequence of pointed metric measure spaces with $\Xdmxn \overset{pmGH}{\rightarrow} \Xdmxinf$. Suppose $(\X_n,\sfd_n,\mm_n)$ satisfies the (BIP) with profile function $D\mapsto C^n_q(D)$ for all $n \in \N$ and there exist non increasing assignments $D\mapsto C_q(D)$ so that $\lims_n C_q^n(D) \le C_q(D)<\infty$ for every $D>0,q \in (1,\infty)$.

Then $(\X_\infty,\sfd_\infty,\mm_\infty)$ has the (BIP) with profile function $D\mapsto C_q(D)$.
\end{theorem}
\begin{proof}

We subdivide the proof in two steps.

\noindent\textsc{Step 1}. Let $q\in(1,\infty)$, $\mu_\infty=\rho_\infty\mm_\infty\in\cP(\X_\infty)\subset\cP(\Z)$ be with bounded support and $A\subset\Z$ open bounded with $\sfd({\rm supp}(\mu_\infty ),\Z\setminus A)>0$. We claim that there is a sequence $n\mapsto\mu_n=\rho_n\mm_n\in\cP(\X_n)\subset\cP(\Z)$ $W_q$-converging to $\mu_\infty$ with 
\begin{equation}
\label{eq:limsdens}
\lims_n\|\rho_n\|_{L^\infty(\mm_n)}\leq \|\rho_\infty\|_{L^\infty(\mm_\infty)}
\end{equation}
such that ${\rm supp}(\mu_n)\subset A$ for every $n$ sufficiently big.

To see this, let $\eta:\Z\to[0,1]$ be continuous, identically 1 on ${\rm supp}(\mu)$ and with support contained in $A$. Put
\[
\tilde\mm_n:=\frac1{z_n}\eta\mm_n,\qquad\text{ where }\qquad z_n:=\int\eta\,\d\mu_n
\]
and similarly $\mm_\infty$. Notice that the assumptions on $\mu_\infty$ ensure that $z_\infty>0$, so that $\tilde\mm_\infty$ is well defined, and thus the pmGH-convergence guarantees that $z_n>0$ for every $n$ sufficiently big, so that for these $n$'s the probability measures $\tilde\mm_n\in\cP(\Z)$ are well defined and weakly converge to $\tilde\mm_\infty$ in duality with $C_b(\Z)$. In the forthcoming discussion we will neglect the small $n$'s and think the $\tilde\mm_n$'s to be defined for every $n\in\N$.

By construction ${\rm supp}(\tilde\mm_n)\subset A$ for every  $n\in\N$ and since $A$ is bounded we deduce that $W_q(\tilde\mm_n,\tilde\mm_\infty)\to 0$ as $n\to\infty$. Let $\alpha^n\in \Opt_q(\tilde{\mm}_\infty, \tilde{\mm}_n)$, and define 
\[
\mu_n:=P^2_\sharp \beta^n,\qquad\text{ where }\qquad \d\beta^n(x,y) := \frac{\d {\mu}_\infty}{\d \tilde\mm_\infty}(x)\,\d\alpha^n(x,y) \in \cP(\Z\times \Z).
\]
Notice that $P^1_\sharp\beta^n=\mu_\infty$, and thus $\beta^n\in \Adm(\tilde{\mu}_\infty,\tilde \mu_n)$. Also, from $\frac{\d {\mu}_\infty}{\d \tilde\mm_\infty}=z_\infty\frac{\d {\mu}_\infty}{\d \mm_\infty}=z_\infty\rho_\infty$ we get $\beta^n \leq z_\infty \|\rho_\infty\|_{L^\infty(\mm_\infty)} \alpha^n$ and thus  
\[
\mu_n \leq z_\infty \|\rho_\infty\|_{L^\infty(\mm_\infty)} P^2_\sharp \alpha^n =z_\infty \|\rho_\infty\|_{L^\infty(\mm_\infty)} \tilde \mm_n\leq\frac{z_\infty}{z_n} \|\rho_\infty\|_{L^\infty(\mm_\infty)} \mm_n.
\]
Since clearly $z_n\to z_\infty$, \eqref{eq:limsdens} holds.  Moreover, we have
\begin{equation}
\begin{split}
W_q^q( \mu_\infty, \mu_n)& \le \int\sfd^q(x,y) \, \d \beta^n(y_1,y_2) \\
&\leq z_\infty \|\rho_\infty\|_{L^\infty(\mm_\infty)}\int\sfd^q(x,y) \, \d \alpha^n(y_1,y_2)\le z_\infty \|\rho_\infty\|_{L^\infty(\mm_\infty)} W_q^q(\tilde\mm_\infty,\tilde \mm_n)\to 0,
\end{split}\label{eq:pushestim}\end{equation}
and the claim is proved.

\noindent\textsc{Step 2}. Let $D>0$, $\mu_0,\mu_1 \in \cP(\X_\infty)$ be absolutely continuous with bounded densities and ${\rm diam}({\rm supp}(\mu_0) \cup{\rm supp}(\mu_1))<D$. Let  $A\subset\Z$ be open with ${\rm diam}(A)<D$ and $\sfd(({\rm supp}(\mu_0) \cup{\rm supp}(\mu_1)),\Z\setminus A)>0$: apply the previous step with $A$ and the measures $\mu_0,\mu_1$ to find corresponding sequences $(\mu^n_i)$ as above. Since $\X_n$ is a BIP space we can find $\pi^n\in \OptGeo_q(\mu^n_0,\mu^n_1)$ with 
\begin{equation}
\label{eq:bipn}
\Comp(\pi^n)\leq C_q^n(D)\big(\|\rho_0^n\|_{L^\infty(\mm_n)}\vee \|\rho^n_1\|_{L^\infty(\mm_n)}\big),
\end{equation}
where $\rho^n_i:=\frac{\d\mu^n_i}{\d\mm}$. 
By Lemma  \ref{lem:pipmGH} above, the sequence $(\pi^n)$ has a subsequence weakly converging to some $\pi\in \OptGeo_q(\mu_0,\mu_1)$, so that taking into account {\rm iii)} of Lemma \ref{lem:pipmGH} and the previous step, by taking the $\lims$ in \eqref{eq:bipn} we conclude that
\[
\Comp(\pi)\leq\lims_{n\to\infty} C_q^n(D)\big(\|\rho_0\|_{L^\infty(\mm_\infty)}\vee \|\rho_1\|_{L^\infty(\mm_\infty)}\big)
\]
and the conclusion follows.
\end{proof}

\section{$p$-independent differential calculus}
\subsection{Unification of $p$-differential calculus}
In this section we study the effect of $p$-independence of weak upper gradients in terms of differential calculus. Informally, the idea is that on this sort of spaces we should have a concept of differential (and thus of cotangent module) which is independent on the chosen Sobolev exponent $p$.
\begin{theorem}[Universal cotangent module - weak version]\label{thm:dpindepenw}
Let $\Xdm$ be a metric measure space with $p$-independent weak upper gradients in weak sense. 

Then there is a unique couple $(L^0(T^*\X),\d)$ where $L^0(T^*\X)$ is a $L^0(\mm)$-normed module and $\d \colon \cup_{p\in(1,\infty)}S^p_{loc}(\X)\rightarrow L^{0}(T^*\X)$ is such that for any $p\in(1,\infty)$ it holds
\begin{itemize}
\item[{\rm i)}] The restriction of $\d$ to $S^p_{loc}(\X)$ is linear;
\item[{\rm ii)}] For any $f\in S^{p}_{loc}(\X)$, it holds $|Df|_p=|\d f|$ $\mm$-a.e.;
\item[{\rm iii)}] The space $\{\d f\colon f \in W^{1,p}(\X)\}$ generates $L^0(T^*\X)$ as a module (for each fixed $p\in(1,\infty)$ - in particular also $\cup_p\{\d f\colon f \in W^{1,p}(\X)\}$ generates $L^0(T^*\X)$).
\end{itemize}
Here, uniqueness is intended up to unique isomorphism, i.e. if $(\MM,L)$ is another couple with the same properties, there is a unique module isomorphism $\Phi \colon \MM\rightarrow L^0(T^*\X)$ so that $\Phi\circ L= \d$.

Moreover, the identification $I_p \colon L^0_p(T^*\X)\to L^0(T^*\X)$ sending $\d_p f\mapsto \d f$ induces the module isomorphism $J_p \colon L^0_p(T\X) \to (L^0(T\X))^*$.
\end{theorem}
\begin{proof}\ \\
\textsc{Uniqueness}. For any $p \in (1,\infty)$, the couple $(L^0(T^*\X),\d\restr{S^{p}_{loc}})$ satisfies the same properties of $(L_p^0(T^*\X),\d_p)$ in Theorem \ref{thm:ctgmodule}. Therefore, uniqueness is a direct consequence of the uniqueness part such Theorem.\newline
\textsc{Existence}. Fix $p _1,p_2 \in (1,\infty), f,g\in W^{1,p_1}(\X)\cap W^{1,p_2}(\X)$ and $E\subset \X$ Borel. We observe that locality of differentials together with the assumption on the metric measure space yield
\begin{equation} \label{eq:localind}
\begin{aligned}
\d_{p_1} f= \d_{p_1}g \quad \mm\text{-a.e. on }E & &\Leftrightarrow & &\d_{p_2} f= \d_{p_2} g  &\quad \mm\text{-a.e. on }E 
\end{aligned}
\end{equation}
Indeed, 
\begin{equation*} 
\begin{aligned}
&\d_{p_1} f= \d_{p_1} g \quad \mm\text{-a.e. on }E & &\Leftrightarrow & & |\d_{p_1}(f-g)|=0  &\mm\text{-a.e. on }E \\
&& & \Leftrightarrow & &|D(f-g)|_{p_1}=0 &\mm\text{-a.e. on }E \\
&& & \Leftrightarrow & &|D(f-g)|_{p_2}=0 &\mm\text{-a.e. on }E \\
&& & \Leftrightarrow & &|\d_{p_2}(f-g)|=0 &\mm\text{-a.e. on }E \\
&& & \Leftrightarrow & &\d_{p_2}f=\d_{p_2}g  &\mm\text{-a.e. on }E 
\end{aligned}
\end{equation*}
Building up on property \eqref{eq:localind} we are going to construct an isomorphism  $I_{p_1}^{p_2} : L^0_{p_1}(T^*\X) \rightarrow L^0_{p_2}(T^*\X)$ sending $\d_{p_1}f$ to $\d_{p_2}f$. We start by defining its action on simple $1$-forms. Denote by $V_{p_j} \subset L^0_{p_j}(T^*\X)$, $j=1,2$, the space of covector fields of type $\sum_{i=1}^n \nchi_{E_i}\d_{p_j} f_i$,  where $(E_i)$ is a finite Borel  partition of $\X$, and $(f_i) \subset W^{1,p_1}(\X)\cap W^{1,p_2}(\X)$. Then define $I_{p_1}^{p_2} \colon V_{p_1}\rightarrow V_{p_2}$ by the formula 
\[ 
I_{p_1}^{p_2}\Big(\sum_{i=1}^n \nchi_{E_i}\d_{p_1} f_i\Big):= \sum_{i=1}^n \nchi_{E_i}\d_{p_2} f_i. 
\]
It can be readily checked that \eqref{eq:localind} ensures the well posedness of such map. Moreover, $I_{p_1}^{p_2}$ is linear and, due to the independence of weak upper gradients, it is a pointwise isometry, since
\[ \Big|  I_{p_1}^{p_2}\Big(\sum_{i=1}^n \nchi_{E_i}\d_{p_1} f_i\Big) \Big| = \sum_{i=1}^n \nchi_{E_i}|\d_{p_2} f_i| = \sum_{i=1}^n \nchi_{E_i}|\d_{p_1} f_i| = \Big| \Big(\sum_{i=1}^n \nchi_{E_i}\d_{p_1} f_i\Big) \Big| \quad \mm\text{-a.e.}\]
Therefore, it is continuous  and, with a little abuse of notation, it uniquely extends to a pointwise isometry from the closure of $V_{p_1}$ with values in $L^0_{p_2}(T^*\X)$. It is clear that, thanks to a) of Definition \ref{def:pindependent} and {\rm ii)} of Theorem \ref{thm:ctgmodule}, the closure of $V_{p_j}$ coincides with $L_{p_j}^0(T^*\X)$ itself, $j=1,2$. We thus built a module isomorphism $I_{p_1}^{p_2}$ such that
\begin{equation}
\label{eq:identificdiff}
I_{p_1}^{p_2}(\d_{p_1} f) = \d_{p_2} f,
\end{equation} 
holds for every $f \in W^{1,p_1}(\X)\cap W^{1,p_2}(\X)$ and it is then clear from  definition of $S^{p_j}_{loc}(\X)$ and the locality of the differential  that \eqref{eq:identificdiff} holds for any $f\in S^{p_1}_{loc}(\X)\cap S^{p_2}_{loc}(\X)$.

To conclude, fix $\bar{p} \in (1,\infty)$, set $I_p := I_p^{\bar p}$, define the module $L^0(T^*\X) := L_{\bar{p}}^0(T^*\X)$, and the differential $\d $ as  $\d\restr{S^{p}_{loc}(\X)} := I_p \circ \d_p$. Notice that \eqref{eq:identificdiff} ensures that the definition of $\d$ is well posed. 

Then property {\rm i)} follows from the linearity of $\d_p$ and $I_p^{\bar p}$. Property {\rm ii)} follows from the fact that  for every $p \in (1,\infty)$, $I_p $ is a pointwise isometry, as $|\d f| = |I_p (\d_p f)|=|\d_pf|=|Df|_p$,  $\mm$-a.e.. Finally, we know that $\{ \d_p f \colon  f \in W^{1,p}(\X)\}$ generates $L^0_p(T^*\X)$, thus {\rm iii)} follows from the fact that $I_p$ is an isomorphism. Finally, the claim about $J_p$ is obvious.
\end{proof}

\begin{definition}[Universal differential structures]\label{def:univ}
Let $\Xdm$ be a metric measure space with $p$-independent weak upper gradients in weak sense. We call the \emph{universal cotangent module} the module $L^0(T^*\X)$ and $\d$ the associated \emph{universal differential} given by Theorem \ref{thm:ctgmodule}. Moreover, we call the \emph{universal tangent module}, and denote it by $L^0(T\X)$, the dual module $(L^0(T^*\X)^*$.
\end{definition}

In spaces with $p$-independent weak gradients in the strong sense, the following stronger result holds. Basically, it says that in this case the situation is closer to the standard Euclidean one, where one first has a distributional differential and then, by investigating its integrability, deduces whether the function is Sobolev:
\begin{theorem}[Universal cotangent module - strong version]\label{thm:dpindepens}
Let $\Xdm$ be a metric measure space with $p$-independent weak upper gradients in strong sense. 

Then in addition to the results in Theorem \ref{thm:dpindepenw}, the following holds. Let $f\in  \cup_{p\in(1,\infty)}S^p_{loc}(\X)$ be such that for some $\bar p\in(1,\infty)$ we have $|\d f|\in L^{\bar p}_{loc}(\mm)$.

Then $f\in S^{\bar p}_{loc}(\X)$.
\end{theorem}
\begin{proof}
By assumption we know that for some $p\in(1,\infty)$ we have $f\in S^p_{loc}(\X)$. For $k,R>0$ let $f^k:=(-k)\vee f\wedge k$, $\eta_R:\X\to[0,1]$ given by $\eta_R:=(1-\sfd(\cdot,B_R(\bar x)))^+$, where $\bar x\in\X$ is some fixed chosen point, and $f^k_R:=\eta_Rf^k$. Then $f^k_R$ is bounded with bounded support, hence it belongs to $L^p\cap L^{\bar p}(\mm)$. Moreover, we have
\[
|D f^k_R|_{p}\leq \eta_R |D f|_{p_1}+\nchi_{{\rm supp}(\eta_R)}k=\eta_R|\d f|+\nchi_{{\rm supp}(\eta_R)}k
\]
and the assumption $|\d f|\in L^{\bar p}_{loc}(\mm)$ gives that the rightmost side in the above is in $L^{\bar p}(\mm)$. Then the fact that $\X$ has $p$-independent weak gradients in the strong sense implies that $f^k_R\in W^{1,\bar p}(\X)$ with $|D f^k_R|_{\bar p}\leq |\d f|+\nchi_{{\rm supp}(\eta_R)}k$. By the very definition of $S^{\bar p}_{loc}(\X)$ we just proved that $f\in S^{\bar p}_{loc}(\X)$, which is the conclusion.
\end{proof}

\subsection{Infinitesimal Hilbertianity and universal gradient}
Intuitively, when a metric measure is infinitesimal Hilbertian, there is a hidden scalar product between tangent directions at small scales. Such geometric concept has in principle nothing to do with the particular choice of the exponent $p=2$. Nevertheless, due to the dependence of weak upper gradients on $p$, this hidden geometry remains unseen by all other $W^{1,p}(\X)$. 

In this section we investigate some consequences of having both infinitesimal Hilbertianity and $p$-independent weak upper gradients. We start with the following proposition, whose proof essentially boils down into verifying that the classical Clarkson inequalities are valid also for elements of a generic Hilbert module:
\begin{proposition}\label{prop:Lpunifconvex}
Let $\Xdm$ be an infinitesimal Hilbertian metric measure space with $p$-independent gradients in weak sense. Then, $L^p(T^*\X)$ and $W^{1,p}(\X)$ are uniformly convex and consequently also reflexive for every $p \in (1,\infty)$.
\end{proposition}
\begin{proof} Reflexivity is  a consequence of uniform convexity, so we focus on this latter property. Also, the map
\[
W^{1,p}(\X)\ni f\quad\mapsto\quad (f,\d f)\in L^p(\X)\times_p L^p(T^*\X)
\]
is an isometry and since $L^p(\X)$ is uniformly convex and so is the $L^p$-norm on $\R^2$ used to define the product norm, the uniform convexity of $W^{1,p}(\X)$ will follow if we show the one of $L^p(T^*\X)$. We thus concentrate on this latter space and observe that it is sufficient to show that the Clarkson inequalities hold:
\begin{equation}
\label{eq:cla}
\begin{array}{lrrrl}
p\in[2,\infty)&\Rightarrow&&\qquad \Big\| \dfrac{\omega+\eta}{2}\Big\|_{L^p(T^*\X)}^p+  \Big\|\dfrac{\omega-\eta}{2}\Big\|_{L^p(T^*\X)}^p &\le \frac{1}{2}\|\eta\|_{L^p(T^*\X)}^p + \frac{1}{2}\|\omega\|_{L^p(T^*\X)}^p,\\
p\in(1,2]&\Rightarrow&&\qquad \Big\| \dfrac{\omega+\eta}{2} \Big\|_{L^p(T^*\X)}^q + \Big\| \dfrac{\omega-\eta}{2} \Big\|_{L^p(T^*\X)}^q& \le \Big( \frac{1}{2} \| \omega \|_{L^p(T^*\X)}^p +\frac{1}{2} \| \eta \|_{L^p(T^*\X)}^p \Big)^\frac{q}{p},
\end{array}
\end{equation}
where $q\in(1,\infty)$ is the conjugate exponent of $p$ and $\omega,\eta$ are arbitrary elements in $L^0(T^*\X)$.

To see these, we start noticing that  the assumption of infinitesimal Hilbertianity (and Theorem \ref{thm:dpindepenw} and its proof) gives that
\begin{equation}
\label{eq:infhiln}
2|\eta|^2+2|\omega|^2=|\eta+\omega|^2+|\eta-\omega|^2,\qquad\mm\text{-a.e.},\qquad\forall \eta,\omega\in L^0(T^*\X).
\end{equation}
\textsc{Case }$p\ge 2$. From the inequality $\|x\|_p\leq \|x\|_2$ valid for any $x\in \R^2$ we see that for any $\eta,\omega\in L^0(T^*\X)$ we have
\begin{equation}
\label{eq:p2}
\Big| \frac{\omega+\eta}{2}\Big|^p+  \Big|\frac{\omega-\eta}{2}\Big|^p \le \Big(\Big| \frac{\omega+\eta}{2}\Big|^2+  \Big|\frac{\omega-\eta}{2}\Big|^2 \Big)^{p/2}\stackrel{\eqref{eq:infhiln}}= \Big( \frac{|\eta|^2}{2}+\frac{|\omega|^2}{2}\Big)^{p/2}  \leq \frac{|\omega^n|^p}{2} + \frac{|\eta^n|^p}{2}, 
\end{equation}
having used the fact that $\R^+\ni t\mapsto  t^{\frac p2}$ is convex in the last step. Integrating we deduce the first in \eqref{eq:cla}.

\textsc{Case } $p\in (1,2]$. Obviously  $p \le 2 \le q$ and thus for any $\eta,\omega\in L^0(T^*\X)$ we have
\begin{equation}
\label{eq:perrev}
\Big| \frac{\omega+\eta}{2}\Big|^q+  \Big|\frac{\omega-\eta}{2}\Big|^q\stackrel{\eqref{eq:p2}}\leq \Big( \frac{|\eta|^2}{2}+\frac{|\omega|^2}{2}\Big)^{q/2}\stackrel{*}\leq   \Big( \frac{|\eta|^p}{2}+\frac{|\omega|^p}{2}\Big)^{q/p},
\end{equation}
where in the starred inequality we used the fact that $\|x\|_2\leq \|x\|_p$  for any $x\in \R^2$.

Now suppose we already know the reverse triangle inequality in $L^r$ spaces for $r\in(0,1)$, i.e.\  that for $f,g$ Borel non-negative it holds
\begin{equation}
\label{eq:revlp}
\Big(\int f^r\,\d\mm\Big)^{\frac1r}+\Big(\int g^r\,\d\mm\Big)^{\frac1r}\leq \Big(\int (f+g)^r\,\d\mm\Big)^{\frac1r}
\end{equation}
and apply it with $r:=\frac pq$, $f:=\frac12|\omega+\eta|^p$ and $g:=\frac12|\omega+\eta|^p$ to obtain
\[ 
\Big(\Big\| \frac{\omega+\eta}{2} \Big\|_{L^p(T^*\X)}^q + \Big\| \frac{\omega-\eta}{2} \Big\|_{L^p(T^*\X)}^q \Big)^{p/q}\leq \int\Big(  \Big| \frac{\omega+\eta}{2}\Big|^q +  \Big|\frac{\omega-\eta}{2}\Big|^q \Big)^{p/q}\, \d \mm .
\]
This and  \eqref{eq:perrev} give the second in \eqref{eq:cla}, thus it remains to prove \eqref{eq:revlp}. Putting $\varphi(x):=(f^r(x),g^r(x))$ and $\Psi(a,b):=(a^{\frac 1r}+b^{\frac 1r})^r$, \eqref{eq:revlp} takes the form
\[
\Psi\Big(\int\varphi\,\d\mm\Big)\leq\int\Psi\circ\varphi\,\d\mm.
\]
Now observe that since $\Psi:\R^2\to\R$ is convex and positively 1-homogeneous we have
\[
\Psi=\sup_{\ell\leq\Psi,\ \ell\text{ linear}}\ell,
\]
therefore
\[
\Psi\Big(\int\varphi\,\d\mm\Big)=\sup_\ell \ell\Big(\int\varphi\,\d\mm\Big)=\sup_\ell \int\ell\circ \varphi\,\d\mm\leq \int\Psi\circ\varphi\,\d\mm
\]
and the conclusion follows.
\end{proof}

In presence of infinitesimal Hilbertianity and independence of $p$-upper gradients in the weak sense, we can naturally define a linear notion of gradient:
\begin{theorem}[Universal gradient]\label{thm:gradqindepen}
Let $\Xdm$ be a infinitesimal Hilbertian metric measure space with $p$-independent weak upper gradients in weak sense. Then there is a unique map $\nabla \colon \cup_{p \in (1,\infty)} S^{p}_{loc}(\X) \to L^0(T\X)$, called \emph{universal gradient}, such that for any $p \in (1,\infty)$ it holds
\begin{itemize}
\item[{\rm i)}] The restriction of $\nabla$ on $S^{p}_{loc}(\X)$ is linear;
\item[{\rm ii)}] For any $f\in S^{p}_{loc}(\X)$, it holds $\d f(\nabla f) = |\nabla f|_*^2 = |\d f|^2$ $\mm$-a.e.;
\item[{\rm iii)}] The space $\{\nabla  f\colon f \in W^{1,p}(\X)\}$ generates $L^0(T\X)$ as a module.
\end{itemize}
\end{theorem}
\begin{proof} The assumption of infinitesimal Hilbertianity (together with Theorem \ref{thm:dpindepenw} and its proof) ensures that $L^0(T^*\X)$ is a Hilbert module: let $\mathcal R:L^0(T^*\X)\to L^0(T\X)$ be the Riesz isomorphism. 

Then from Theorem \ref{thm:dpindepenw} above it is clear that $\nabla f:=\mathcal R(\d f)$ satisfies the requirements.
\end{proof}

\section{Spaces with the (BIP) and consequences}\label{Sec5}
In this section we list some spaces and conditions that we can relate to the bounded interpolation property and we analyze the consequences. 

\subsection{Non branching MCP-spaces}
We start recalling the definition of distortion coefficient. For every $K\in(0,\infty), N\in (0,\infty), t \in [0,1]$ set
\[
\sigma^{(t)}_{K,N}(\theta) := \begin{cases} 
\ +\infty, &\text{if } K\theta^2 \ge N\pi^2,\\
\ \frac{\sin(t\theta\sqrt{K/N})}{\sin(\theta\sqrt{K/N)}}, &\text{if } 0<K\theta^2<N\pi^2,\\
\ t, &\text{if } K\theta^2<0\text{ and } N=0 \text{ or if } K\theta^2=0, \\
\ \frac{\sinh(t\theta\sqrt{-K/N})}{\sinh(\theta\sqrt{-K/N)}}, &\text{if } K\theta^2\le 0 \text{ and }N>0.\\
\end{cases}
\]
Set, for $N>1$, $\tau^{(t)}_{K,N}(\theta) := t^{\frac{1}{N}}\sigma^{(t)}_{K,N-1}(\theta)^{1-\frac{1}{N}}$ while $\tau^{(t)}_{K,1}(\theta) =t$ if $K\le 0$ and $\tau^{(t)}_{K,1}(\theta)=\infty$ if $K>0$.
Finally, for $\mu \in \cP(\X)$ and $N \in [1,\infty)$, we define the $N$-\emph{R\'enyi relative entropy} with respect to $\mm$ by 
\[ \mathcal{U}_N(\mu | \mm) := -\int \rho^{1-\frac{1}{N}}\, \d \mm, \qquad \text{if } \mu = \rho \mm+\mu^s,\quad\mu^s\perp\mm,\]
and the \emph{Shannon entropy} by
\[ {\rm Ent}_\mm(\mu) := \int \rho\log \rho \,\d \mm, \qquad \text{if} \mu = \rho\mm, \quad \infty\text{ otherwise}. \]

Here, we report the definition of the \emph{measure contraction property}, as introduced independently in \cite{Ohta07} and \cite{Sturm06II}.
\begin{definition}[$\mcp$-spaces]\label{def:MCP}
We say that a metric measure space $\Xdm$ satisfies the \emph{measure contraction property} $\mcp(K,N)$ for some $K \in \R, N \in [1,\infty)$ if for any $\mu_0=\rho_0\mm \in \cP_2(\X)$ absolutely continuous with bounded support contained in $\supp(\mm)$ and $o \in \supp(\mm)$, there exists $\pi \in \OptGeo_2(\mu_0,\delta_o)$ so that
\[ \mathcal{U}_N(\mu_t) \le - \int \tau^{(1-t)}_{K,N}(\sfd(x,o))\rho_0^{1-\frac 1N}\, \d \mm,\qquad \forall t \in [0,1),\]
having set $\mu_t:= (\e_t)_\sharp \pi$.
\end{definition}

In this note we will couple the $\mcp$-class with the \emph{non branching} condition: a geodesic metric space $(\X,\sfd)$ is called non branching, provided
\[ \gamma,\theta \text{ geodesic with } \gamma\restr{[0,t]}=\theta\restr{[0,t]}  \text{ for some }t\in(0,1) \quad \Rightarrow \quad \gamma \equiv \theta \text{ on }[0,1].\]
The reason is that the non branching $\mcp$-class enjoys good properties of Wasserstein geodesics that in turns implies deep analytical consequences. We start by recalling that the two definitions of the $\mcp$-condition given in \cite{Ohta07} and \cite{Sturm06II} coincide under the non branching assumption (actually, under the weaker $2$-essentially non branching assumption, a technical property investigated in \cite{RajalaSturm12} that we shall never employ in this note). Then, a non branching $\mcp(K,N)$ metric measure space $\Xdm$ is locally uniformly doubling and supports a weak local $(1,1)$-Poincar\'e inequality \cite{Ohta07,vRS08}. In particular, it is proper. 

\begin{remark}\label{rem:MCPindepend}
\rm Notice that Definition \ref{def:MCP} is \emph{independent} on the particular choice $q=2$. The first observation is that, if $o \in \supp(\mm)$, the set $\OptGeo_q(\mu_0,\delta_o)$ is independent on $q$ when $\mu_0=\rho_0\mm$ with $\rho_0 \in L^\infty(\mm)$ and of bounded support. Indeed, in this case
\[ \mu_0 \in \cP_q(\X)\qquad \text{and}\qquad W_q^q(\mu_0,\delta_o) = \int \sfd^q(x,o)\, \d \mu_0(x), \quad \forall q \in(1,\infty).\]
The verification being that the plan $\mu_0 \otimes \delta_0$ is the only admissible coupling between the two marginal, and therefore it must be optimal for any $q$. Then, it is straightforward to see that if $\pi \in\OptGeo_2(\mu_0,\delta_0)$, then $(\e_1)_\sharp \pi = \delta_o$ and therefore $\pi \in \OptGeo_q(\mu_0,\delta_0)$ for all $q \in (1,\infty)$. This discussion automatically shows that the choice of $q=2$ in Definition \ref{def:MCP} is irrelevant.
\fr
\end{remark} 

Finally, we state and prove an extension to abritrary $q$ of a result present in \cite{CavMon17} (see also \cite{Kell17}). This will directly imply the (BIP) condition.
\begin{theorem}\label{thm:MCPestim}
Let $\Xdm$ be a non branching $\mcp(K,N)$-space for some $K \in \R, N \in [1,\infty)$. Then, for every $q \in (1,\infty)$, $D>0$ and $\mu_0,\mu_1 \in \cP_q(\X)$ with $\mu_0=\rho_0\mm, \rho_0 \in L^\infty(\mm)$, and diam$(\supp(\mu_0)\cup\supp(\mu_1))<D$, there exists $\pi \in \OptGeo_q(\mu_0,\mu_1)$ with $\mu_t:=(\e_t)_\sharp \pi \ll \mm$ and
\begin{equation} \|\rho_t\|_{L^\infty(\mm)} \le \frac{1}{(1-t)^N}e^{Dt\sqrt{(N-1)K^-}}\| \rho_0\|_{L^\infty(\mm)},\qquad \forall t \in [0,1), \label{eq:MCPrhot}
\end{equation}
having set $\rho_t := \frac{\d \mu_t}{\d \mm}$ for $t<1$.
\end{theorem}
\begin{proof}
For $q=2$, the statement is proved in \cite[Theorem 1.1]{CavMon17}. Here, we give some details to handle the general case.

\noindent{\sc Step 1}. We begin by showing that any $\Xdm$ as in the hypothesis is `qualitatively non degenerate' according to the axiomatization given in \cite[Assumption 1]{CM15} (actually, under non branching it is equivalent \cite[Corollary 5.17]{Kell17}). Indeed, let $K\subset \X$ be compact, $A\subset K$ non negligible and $x \in K$. Then, denoting by $A_{t,x} \subset C([0,1],\X)$ the subset of geodesics linking $A$ to $x$, we denote $\mu_0 =\mm(A)^{-1}\mm\restr{A}$ and appeal to \cite[Theorem 1.1]{CavMon17} to get that there exists $\pi \in \OptGeo_2(\mu_0,\delta_x)$ so that
\[ 1=\mu_t(A_{t,x}) \le \frac{1}{(1-t)^N}e^{Dt\sqrt{(N-1)K^-}}\frac{\mm(A_{t,x})}{\mm(A)}, \qquad \forall t \in [0,1),\]
where $D:= $ diam $(K) <\infty $, having used that $\mu_t$ is a probability measure concentrated on $A_{t,x}$. In particular, this shows that there exists a profile function $f \colon [0,1]\to (0,1]$ (depending on the compact set $K$) and a positive $\delta < 1$ so that
\[ \mm(A_{t,x}) \ge f(t) \mm(A), \qquad \forall t \in [0,\delta].\]
That is, $\Xdm$ verifies Assumption $1$ in \cite{CM15}.

\noindent{\sc Step 2}. Suppose now $\mu_1= \delta_x$, for $x \in \supp(\mm)$. Then, from \cite{CavMon17}, we know that there exists $\pi \in \OptGeo_2(\mu_0,\delta_x)$ satisfying \eqref{eq:MCPrhot} hence, by Remark \ref{rem:MCPindepend}, we have also that $\pi \in \OptGeo_q(\mu_0,\delta_x)$.

\noindent{\sc Step 3}. Let here $n \in \N$ and suppose $\mu_1$ is a finite convex combination of Dirac masses, namely $\mu_1 := \sum_{j=1}^n\lambda_j \delta_{x_j}$ for $(x_j)\subset \X $ with $x_i\neq x_i$ for $i\neq j$ and $(\lambda_j)\subset [0,1]$ with $\sum_{j=1}^n\lambda_j=1$. Then, by appealing to {\sc Step 1}, we are in position to apply \cite[Theorem 2.1]{CM15} (recall $\X$ is proper) to deduce that there exists a unique optimal coupling between $\mu_0$ and $\mu_1$ and it is induced by a Borel map $T \colon \X\to \X$, i.e.
\[ W_q^q(\mu_0,\mu_1) = \int \sfd^q(x,T(x))\, \d \mu_0(x). \]
Take this map $T$ and define $\mu_0^j := \mu_0 \restr{T^{-1}(x_j)}$ for every $j =1,...,n$. By {\sc Step 2}, we know that there are $\pi^j \in \OptGeo_q(\lambda_j^{-1}\mu_0^j,\delta_{x_j})$ verifying 
\[ \|\rho^j_t\|_{L^\infty(\mm)} \le \frac{1}{(1-t)^N}e^{Dt\sqrt{(N-1)K^-}}\| \rho^j_0\|_{L^\infty(\mm)},\qquad \forall t \in [0,1),\  j=1,...,n,
\]
having set $\rho^j_t := \frac{\d \mu^j_t}{\d \mm}$ and $\mu_t^j := (\e_t)_\sharp \pi^j$ for $t<1$.

We now define $\pi := \sum_{j=1}^n \lambda_j\pi^j$ so that, by construction, we have that $(\e_0,\e_1)_\sharp \pi \in \Adm(\mu_0,\mu_1)$ and $\mu_t:= (\e_t)_\sharp \pi \ll \mm$ for every $t <1$ with $\rho_t:= \frac{\d \mu_t}{\d \mm}  = \sum_{j=1}^n \lambda_j \rho^j_t $. We now claim that $\rho_t$ satisfies \eqref{eq:MCPrhot}. To this aim, we instead check that
\[ \mm(\{ \rho^i_t>0 \} \cap \{ \rho^j_t>0\} )) =0,\qquad \forall t \in (0,1), j\neq i,\]
as the latter property implies the claim by construction of $\mu_t$. Suppose the above is not true, namely there exists $\tau \in (0,1)$ so that $\mm(\{ \rho^i_\tau>0 \} \cap \{ \rho^j_\tau>0\} )) >0$ for some $j\neq i$. Then there would exist an optimal coupling between the (renormalized) measures $\mm\restr{\{ \rho^i_\tau>0 \} \cap \{ \rho^j_\tau>0\} }$ and $\delta_{x_j}+\delta_{x_i}$ that is not induced by a map. This finds a contradiction for what we have previously proved and concludes the step.

\noindent{\sc Step 4}. Since $(\X,\sfd)$ is separable, given any $\mu_1$ as in the hypothesis, we can find a sequence of points $(x_j)_{j \in \N}\subset \supp(\mu_1)$ and weights $(\lambda_{n,j})_{j \in \N}\subset [0,1]$ so that
\[ \sum_{j=1}^n\lambda_{n,j}\delta_{x_j} =:\mu_1^n \to \mu_1 \qquad \text{in } W_q,\]
as $n$ goes to infinity, recalling that the support of $\mu_1$ is bounded and consequently \eqref{eq:Wqweak}. By {\sc Step 3}, we know that there exists a sequence $\pi_n \in \OptGeo_q(\mu_0,\mu_1^n)$ verifying $\mu^n := \rho_t^n\mm$ and \eqref{eq:MCPrhot} for every $t <1, n \in \N$. Finally, arguing as in \cite[Lemma 4.4]{CavMon17} (the proof is written for $q=2$ but works for arbitrary $q$, see also Lemma \ref{lem:pipmGH}) we get the existence of a weak limit $\pi \in \OptGeo_q(\mu_0,\mu_1)$ of $\pi_n$ verifying all the required properties.
\end{proof}
%
%
%
A direct implication of the above Theorem is the following:
\begin{theorem}	\label{thm:BIPMCP}
Let $\Xdm$ be a non branching $\mcp(K,N)$-space for some $K \in \R, N \in [1,\infty)$. Then, it is a BIP-space with profile function $D\mapsto 2^Ne^{D\sqrt{(N-1)K^-}}$ and consequently:

\begin{itemize}
\item[{\rm i)}] $\Xdm$ has $p$-independent weak upper gradients in the strong sense;
\item[{\rm ii)}] there exists a unique couple $(L^0(T^*\X),\d)$ of universal cotangent module and differential;
\item[{\rm iii)}] there exists an $\infty$-plan $\ppi_{\sf master}$ concentrated on geodesics so that: 

for every $p \in (1,\infty)$, $f$ Borel and $G\in L^p(\mm)$ with $G\ge 0$, the following are equivalent
\begin{itemize}
\item[$\circ$] $f \in S^p(\X)$ and $G$ is a $p$-weak upper gradient;
\item[$\circ$] it holds
\[ |f(\gamma_1)-f(\gamma_0)| \le \int_0^1 G(\gamma_t)\msp \, \d t,\qquad \ppi_{\sf master}\text{-a.e. }\gamma.\]
\end{itemize}
\end{itemize} 
\end{theorem}
\begin{proof}
We subdivide the proof in two steps.

\noindent\textsc{Step 1}. Let $q \in (1,\infty),D>0$ and $\mu_0,\mu_1 \in \cP_q(\X)$ as in the definition of the (BIP). Then, Theorem \ref{thm:MCPestim} ensures that we can find $\pi^+ \in \OptGeo_q(\mu_0,\mu_1)$ so that
\[ \|\rho^+_t\|_{L^\infty(\mm)} \le 2^N e^{D\sqrt{(N-1)K^-}}\| \rho_0\|_{L^\infty(\mm)},\qquad \forall t \in [0,\tfrac 12], \] 
and $\pi^-\in \OptGeo_q(\mu_1,\mu_0)$ so that 
\[ \|\rho^-_t\|_{L^\infty(\mm)} \le 2^N e^{D\sqrt{(N-1)K^-}}\| \rho_1\|_{L^\infty(\mm)},\qquad \forall t \in [0,\tfrac 12], \]
having denoted $\rho_t^\pm := \frac{\d (\e_t)_\sharp \pi^\pm}{\d \mm}$ for every $t \in [0,\frac 12]$. Hence, by nonbranching and arguing as in \cite{CavMon17}, the set $\OptGeo_q(\mu_0,\mu_1)$ is a singleton and therefore we can glue the \emph{forward} and \emph{backward} estimates to verify the (BIP).

\noindent\textsc{Step 2}. Finally, recalling that $\mcp$-spaces are doubling, we have thanks to \cite{ACM14} that $W^{1,p}(\X)$ is reflexive and therefore $\Lip_{bs}(\X)$ is dense in every $W^{1,p}(\X)$. Consequently, $\Xdm$ have $p$-independent weak upper gradient in the strong sense by appealing to Theorem \ref{thm:pind}. The existence of a universal cotangent module $L^0(T^*\X)$ is guaranteed by Theorem \ref{thm:dpindepenw}. Finally, {\rm iii)} is the content of Theorem \ref{A3}.
\end{proof}
%

\begin{remark}\rm
Since it is known that $\cd(K,N)$ spaces are also $\mcp(K,N)$, the conclusions of this last theorem hold in particular for  non branching $\cd(K,N)$ spaces, $N<\infty$.
\fr\end{remark}

\subsection{$\cd(K,\infty)$-spaces} The arguments of the last section do not cover the case of $\cd(K,\infty)$ spaces and actually we do not know whether the conclusions of Theorem \ref{thm:BIPMCP} hold in this case or not, due to the fact that it is unknown whether the $\cd(K,\infty)$ condition implies anything on the structure of plans in $\OptGeo_q(\mu_0,\mu_1)$ for $q\neq 2$.

Still, for $q\in(1,\infty)$ and $K\in\R$ we can consider the class of $\cd_q(K,\infty)$ spaces, defined as those spaces $(\X,\sfd,\mm)$ such that for any $\mu_0,\mu_1\in \cP_q(\X)$ absolutely continuous w.r.t.\ $\mm$ and with bounded support, there is a $W_q$-geodesic $(\mu_t)$ connecting them such that
\[
{\rm Ent}_\mm(\mu_t)\leq(1-t){\rm Ent}_\mm(\mu_t)+t{\rm Ent}_\mm(\mu_1)-\frac K2t(1-t)W_q^2(\mu_0,\mu_1)\qquad\forall t\in[0,1].
\] 
This class of spaces is relevant for our purposes because of the following result, that generalizes  the ones  in \cite{Rajala12-2} to the case of general $q\ne 2$.
\begin{theorem}\label{thm:rajalaCDqKN}
Let $\Xdm$ be a $\cd_q(K,\infty)$-space for some $K \in \R$ and $q \in (1,\infty)$. For any $D>0$ and $\rho_0 ,\rho_1 \in L^{\infty}(\mm)$ probability densities with diam$($supp$(\rho_0) \ \cup\ $supp$(\rho_1))<D$, there exists $\pi \in \OptGeo_q(\rho_0\mm,\rho_1\mm)$ satisfying $\mu_t:=(\e_t)_\sharp \pi \ll \mm$. Moreover, writing $\mu_t:= \rho_t\mm$, we have
\[ \|\rho_t\|_{L^\infty(\mm)} \le e^{K^-D^2/12} \|\rho_0\|_{L^\infty(\mm)}\vee \|\rho_1\|_{L^\infty(\mm)},\qquad \forall t \in [0,1].\]
\end{theorem}
The proof of this Theorem closely follows the one  in \cite{Rajala12-2} hence, not to interrupt the line of thoughts at this point, we postpone its proof to Appendix \ref{B}.

It is not known if the $\cd_q(K,\infty)$-class is independent on $q$ (this is instead the case in finite dimension with technical assumptions \cite{ACMCS21}), still as a consequence of the previous considerations we have  the following result:
\begin{theorem}
Let $K \in \R$ and suppose that $\Xdm$ is a $\cd_q(K,\infty)$-space for every $q \in (1,\infty)$. Then, it has the (BIP) with profile function $D\mapsto e^{K^-D^2/12}$.
\end{theorem}
\begin{proof}
Observe that the hypotheses with Theorem \ref{thm:rajalaCDqKN} ensure the (BIP) to hold.
\end{proof}

\subsection{$\rcd$-spaces}
In this section, we restrict our attention to the $\rcd$-setting (\cite{AmbrosioGigliSavare11-2,Gigli12}). We start with the definition of this class.
\begin{definition}[$\rcd$-spaces]\label{def:RCD}
Let $K\in \R$ and $N \in [1,\infty]$. A metric measure space $\Xdm$ is a $\rcd(K,N)$-space, provided it is an infinitesimal Hilbertian $\cd_2(K,N)$-space.
\end{definition} 
For a reason that we now make clear, we avoided introducing the exponent's subscript in the definition of $\rcd$-space (i.e. a $\rcd_2$-condition), even though the two requirements characterizing the above definition are strictly related to the exponent $p=2$. Indeed, the first motivation is that, recently in \cite{Deng20}, it has been proven that $\rcd(K,N)$-spaces with $N<\infty$ are non-branching. Therefore, by appealing to the work \cite{ACMCS21}, it is clear that \emph{a posteriori} in Definition \ref{def:RCD} the choice $p=2$ is irrelevant (to be precise, in \cite{ACMCS21} the results are stated only for spaces of finite mass - still it is expected that they generalize to the setting of $\sigma$-finite spaces).  While, the second motivation is that in the work \cite{GigliHan14} it has been already shown that $\rcd(K,\infty)$-spaces posses $p$-independent weak upper gradients. Again, this implies \emph{a posteriori} that one can equivalently require the defining parallelogram rule with arbitrary minimal $p$-weak upper gradients. The reason for which the arguments in \cite{GigliHan14}, contrary to the present note, do not involve other exponents than $p=2$ is that, using heat-flow regularization techniques (which are well understood and at hand in this class, see e.g. \cite{AmbrosioGigliSavare11}), the problem of independence is reduced to the Lipschitz class which is \emph{large} enough in the $\rcd$-setting in every Sobolev space $W^{1,p}(\X)$ and capable to lead to the conclusion of $p$-independent weak upper gradient even in the infinite-dimensional setting of $\rcd(K,\infty)$-spaces. For the sake of completeness, we prove briefly which kind of independence is implied by \cite{GigliHan14} according to the axiomatization of Definition \ref{def:pindependent}.
\begin{proposition}\label{prop:RCDpindep}
Let $\Xdm$ be a $\rcd(K,N)$-space for some $K \in \R, N \in [1,\infty]$. Then, $\Xdm$ has $p$-independent weak upper gradient in the strong sense. Moreover, there are unique couples $(L^0(T^*\X),\sfd)$ and $(L^0(T\X),\nabla)$ of universal cotangent and tangent module, with associated linear universal differential and gradient, respectively.
\end{proposition}
\begin{proof}
It suffices to prove the first claim for $N=\infty$. From \cite{GigliHan14} we know that $\rcd(K,\infty)$-spaces do satisfy $b)$ and $c)$ of Definition \ref{def:pindependent}. It remains to show a). Call now $B:= \overline{\Lip_{bs}(\X)}^{W^{1,p}}$ and observe that, since $\Lip_{bs}(\X) \subset W^{1,p}(\X)\cap W^{1,2}(\X)$ and
\[ 2|Df|_p^2 +2|Dg|_p^2 = |D(f+g)|_p^2+|D(f-g)|_p^2, \qquad \mm\text{-a.e.}, \forall f,g \in \Lip_{bs}(\X), \]
one can argue as in Proposition \ref{prop:Lpunifconvex} in order to show that $B$ is uniformly convex. Then, for every $f \in W^{1,p}(\X)$, we consider thanks to \cite{AmbrosioGigliSavare11-3} and a truncation argument, a sequence $(f_n) \subset \Lip_{bs}(\X)$ so that 
\begin{equation} f_n\rightarrow f, \quad |Df_n|_p \rightarrow  |Df|_p,  \quad \text{in } L^p(\mm). \label{eq:wlscdf}
\end{equation}
Then, for any $G$ weak limit of $\frac{|D(f_n+f_m)|_p}{2}$, by lower semicontinuity of the $W^{1,p}$-norm, we have
\[ \liminf_{n,m}\Big\| \frac{f_n+f_m}{2}\Big\|_{W^{1,p}(\X)} \ge \big(\|f\|^p_{L^p(\mm)}+\|G\|^p_{L^p(\mm)}\big)^{1/p} \ge \|f\|_{W^{1,p}(\X)}.\]
Suppose, by contradiction, that $(f_n)$ is not Cauchy. Then, $\exists \epsilon>0$ so that $\|f_n-f_m\| \ge \epsilon$ for countably many $n,m$. Observe that, uniform convexity of $B$ ensures that $\exists \delta>0$ such that
\[ \|f\|_{W^{1,p}(\X)}-\delta \ge \Big\| \frac{f_n+f_m}{2}\Big\|_{W^{1,p}(\X)},  \]
for countably many $n,m$, which clearly is absurd in light of \eqref{eq:wlscdf}. Therefore, $(f_n)$ is Cauchy and, its $B$-limit must be $f$. In other words, we showed that $\Lip_{bs}(\X)$ is a dense collection in  $W^{1,p}(\X)$ for every $p \in (1,\infty)$, which is a stronger statement implying a). Finally, for the last claim, simply invoke Theorem \ref{thm:gradqindepen} and Theorem \ref{thm:dpindepenw}.
\end{proof}

\subsection{$\cd(K,N)$ with $N<0$}
In this section, we consider a notion of the curvature dimension condition for metric measure spaces with generalized \emph{negative} dimension $N<0$. This was first introduced in \cite{Ohta16} and recently studied in \cite{MRS21,MR21} where the authors considered a larger class of metric measure spaces equipped with a \emph{quasi Radon} reference measures. In the present note we are not interested in working in full generality and recall the definition of this $\cd$-class sticking to our notion of $\Xdm$. Our goal is to show that very naturally, the results of \cite{Rajala12-2} and Appendix \ref{B} extend also to this class.

First, we need to modify the \emph{distortion coefficient} for the case $N\in (-\infty, 0)$ and $K \in \R, t \in [0,1]$:
\[
^-\sigma_{K,N}^{(t)}(\theta) := \begin{cases}
\ +\infty, &\text{if } K\theta^2 \le N\pi^2 ,\\
\ \frac{\sin(t\theta\sqrt{K/N})}{\sin(\theta\sqrt{K/N})} , &\text{if } N\pi^2 < K\theta^2 <0,\\
\ t, &\text{if } K\theta^2=0, \\
\ \frac{\sinh(t\theta\sqrt{-K/N})}{\sinh(\theta\sqrt{-K/N})}, &\text{if } 0<K\theta^2.\\
\end{cases}
\]
Set also $^-\tau^{(t)}_{K,N}(\theta) := t^{\frac{1}{N}}\sigma^{(t)}_{K,N-1}(\theta)^{1-\frac{1}{N}}$. Finally, for $\mu \in \cP(\X)$, we define the $N$-\emph{R\'enyi relative entropy} with respect to $\mm$ with Negative $N$ by 
\[ ^-\mathcal{U}_N(\mu | \mm) := \int_\X \rho^{1-\frac{1}{N}}\, \d \mm, \qquad \text{if } \mu = \rho \mm,\quad \infty\text{ otherwise}.\]

\begin{definition}[{$\cd_q(K,N)$ with }$N<0$]
Let $q \in (1,\infty)$, $K \in \R$ and $N<0$. We say that a metric measure space $\Xdm$ satisfies the \emph{curvature dimension condition} $\cd_q(K,N)$ if any pair of probabilities $\mu_0=\rho_0\mm, \mu_1=\rho_1\mm \in \cP_q(\X)$ admits a plan $\pi \in \OptGeo_q(\rho_0\mm,\rho_1\mm)$ so that
\begin{equation}
^-\mathcal{U}_{N'}(\mu_t|\mm) \le \int \ ^-\tau^{(1-t)}_{K,N'}(\sfd(\gamma_1,\gamma_0))\rho_0(\gamma_0)^{-\frac{1}{N}} + ^-\tau^{(t)}_{K,N'}(\sfd(\gamma_1,\gamma_0))\rho_1(\gamma_1))^{-\frac{1}{N}} \, \d \pi(\gamma),\label{eq:CDKNegativ}
\end{equation} 
for every $t \in [0,1], N' \in [N,0)$ and having denoted $\mu_t := (\e_t)_{\sharp}\pi$.
\end{definition}
It is rather obvious, see e.g. \cite[Proposition 2.6]{MRS21}, that, if $\Xdm$ satisfies the $\cd_q(K,N)$-condition $K \in \R,N<0$, then it satisfies the  $\cd_q(K',N')$-condition for every $K'\le K, N'\in [N,0)$.
\begin{remark}\rm
In order to avoid confusion with the previous definition, we considered writing $^-\sigma,^-\tau,^-\mathcal{U}$ intentionally to distinguish the notation when $N<0$. 
Moreover, we point out that it is natural to consider defining the entropy as $\int \rho^{1-\frac{1}{N}}\, \d \mm$ without a minus sign in front of the integral, as $h(s):= s^{1-\frac{1}{N}}$ is a convex function for $N<0$. 
\fr
\end{remark}
Following \cite{Rajala12-2} and Appendix \ref{B}, we show that also in this case, the curvature dimension condition spreads the support of the measure.
\begin{lemma}\label{lem:spreadNegative}
Let $\Xdm$ be a metric measure space that is a $\cd_q(K,N)$-space for some $K \in R$, $N<0$ and $q \in (1,\infty)$. Then, for any $\rho_0,\rho_1 \in L^\infty(\mm)$ probability densities with $D:=$ diam$(supp(\rho_0)\cup supp(\rho_1)))<\infty$ , we have:

\begin{itemize}
\item[{\rm i)}] If $K\ge 0$, there exists $\pi \in \OptGeo_q(\rho_0\mm,\rho_1\mm)$ so that
\[ \mm(\{\rho_\frac12 >0\}) \ge \frac{1}{\|\rho_0\|_{L^\infty}\vee \|\rho_1\|_{L^\infty}}e^{-\frac{1}{2}\sqrt{(1-N)K}D};\]
\item[{\rm ii)}] If $K<0$ and $D < \pi \sqrt{\frac {N-1}K}$, there exists $\pi \in \OptGeo_q(\rho_0\mm,\rho_1\mm)$ so that
\[ \mm(\{\rho_\frac12 >0\}) \ge \frac{1}{\|\rho_0\|_{L^\infty}\vee \|\rho_1\|_{L^\infty}} \cos^{1-N}(\tfrac 12 D \sqrt{K/N-1});\]
\end{itemize}
where $(\e_{\frac{1}{2}})_{\sharp}\pi = \rho_{\frac{1}{2}}\mm + \mu_{\frac12}^s$ with $\mu_{\frac12}^s\perp \mm$.
\end{lemma}
\begin{proof}
Fix $\pi \in {\rm OpGeo}(\rho_0\mm,\rho_1\mm)$ satisfying \eqref{eq:CDKNegativ}, denote $E:=\{ \rho_{\frac{1}{2}}>0\}$ and notice that $\sfd(\gamma_0,\gamma_1)<D$ $\pi$-a.e. $\gamma$.

Suppose first that $K\ge 0$ and estimate
\[
^-\sigma^{(\frac12)}_{K,N}(\sfd(\gamma_0,\gamma_1)) = \frac{1}{e^{-\frac{1}{2}\sqrt{-K/N}\sfd(\gamma_0,\gamma_1)}+e^{\frac{1}{2}\sqrt{-K/N}\sfd(\gamma_0,\gamma_1)}} \le \frac12e^{\frac{1}{2}\sqrt{-K/N}D}.
\]
so that $^-\tau^{(\frac12)}_{K,N}(\sfd(\gamma_0,\gamma_1)) \le \tfrac12 \big(e^{\frac12\sqrt{(1-N)K}D} \big)^{-\frac 1N}$ $\pi$-a.e. $\gamma$ and lastly
\begin{equation}
\mathcal{U}_N(\mu_\frac12|\mm)  \le  (\|\rho_0\|_{L^\infty}\vee \|\rho_1\|_{L^\infty})^{-\frac{1}{N}}\big(e^{\frac12\sqrt{K(1-N)}D} \big)^{-\frac 1N}.
\label{eq:Unle}
\end{equation}
Moreover, being $1-1/N \ge 1$, an application of Jensen inequality yields the following estimate from below
\begin{equation}
\mathcal{U}_N(\mu_\frac12|\mm) = \mm(E)\fint_E \rho_\frac12^{1-\frac{1}{N}}\,\d \mm \ge \mm(E)\big(\frac{1}{\mm(E)}\big)^{1-\frac{1}{N}} = \mm(E)^{\frac{1}{N}}.
\label{eq:Unge}
\end{equation}
We can now combine \eqref{eq:Unge} with \eqref{eq:Unle}, raise to the $-N$ power and rearrange to conclude in the case $K\ge0$. 

Suppose now instead $K<0$ and $D<\pi\sqrt{N-1/K}$. Then, $\pi$-a.e. $\gamma$ we have that $\frac 12 \sfd(\gamma_0,\gamma_1) \sqrt{K/N-1}) < \pi/2$ and, since the cosine is monotone decreasing and strictly positive on $[0,\pi/2)$, we estimate
\[^-\tau^{(\frac12)}_{K,N}(\sfd(\gamma_0,\gamma_1)) = \frac 12  \Big( \frac{1}{\cos(\frac 12 \sfd(\gamma_0,\gamma_1) \sqrt{K/N-1})} \Big)^{1-\frac 1N} \le \frac 12  \Big( \frac{1}{\cos(\frac 12 D \sqrt{K/N-1})} \Big)^{1-\frac 1N}.\]
Then, combining with \eqref{eq:Unge}, we achieve
\[ \mm(E)^{\frac 1N} \le  (\|\rho_0\|_{L^\infty}\vee \|\rho_1\|_{L^\infty})^{-\frac{1}{N}} \big( \cos(\tfrac 12 D \sqrt{K/N-1}) \big)^{\frac {1-N}N} ,\]
that easily implies the conclusion.
\end{proof} 
From this, it follows:
\begin{theorem}\label{thm:rajalaCDqKNegative}
Let $\Xdm$ be a metric measure space that is a $\cd_q(K,N)$-space for some $K \in \R, N\in (-\infty,0),$ and $q \in (1,\infty)$. Then, for any $D>0$ and $\rho_0 ,\rho_1 \in L^{\infty}(\mm)$ probability densities with diam$($supp$(\rho_0) \ \cup\ $supp$(\rho_1))<D$, it holds
\begin{itemize}
\item[{\rm i)}] if $K\ge 0$, there exists $\pi \in \OptGeo_q(\rho_0\mm,\rho_1\mm)$ with $\mu_t:=(\e_t)_\sharp \pi \ll \mm$ and
\[ \|\rho_t\|_{L^\infty(\mm)} \le  \|\rho_0\|_{L^\infty(\mm)}\vee \|\rho_1\|_{L^\infty(\mm)};\]
\item[{\rm ii)}] If $K<0$ and $D\le \emph{ diam}(\X)<\pi \sqrt{\frac {N-1}K}$, there exists $\pi \in \OptGeo_q(\rho_0\mm,\rho_1\mm)$ with $\mu_t:=(\e_t)_\sharp \ll \mm$ and
\[ \|\rho_t\|_{L^\infty(\mm)} \le   \frac{ \big(\tfrac{D}{4}\sqrt{\tfrac K{N-1} }\big)^{1-N}}{ \sin^{1-N}  \big( \tfrac{D}{4}\sqrt{ \tfrac{K}{N-1}}\big) }  \|\rho_0\|_{L^\infty(\mm)}\vee \|\rho_1\|_{L^\infty(\mm)};\]
\end{itemize}
for all $t \in [0,1]$ and having set $\rho_t:= \frac{\d \mu_t}{\d \mm}$.
\end{theorem}
\begin{proof}
The proof follows directly from the proof of Theorem \ref{thm:rajalaCDqKN} in Appendix \ref{B}, by replacing \textsc{Step 1} there with Lemma \ref{lem:spreadNegative} and repeating \emph{verbatim} \textsc{Step 2} - \textsc{Step 3} - \textsc{Step 4}. For the case {\rm i)}, we simplified the estimate directly working in the larger $\cd_q(0,N)$-class. For the case {\rm ii)}, when $K<0$, we shall also make use (to get the $L^\infty$-bound by completion and induction) of the identity:
\[ \lim_{n\to \infty} \prod_{i=1}^n \cos( \theta 2^{-i}) = \frac{ \sin( \theta  )}{\theta }, \qquad \text{ for }\theta = \tfrac{D}{4}\sqrt{ \tfrac K{N-1}},\]
proven in Lemma \ref{lem:identity} below.
\end{proof}
We finish this section by proving two straightforward corollaries.
\begin{corollary}
Let $\Xdm$ be a metric measure space that is a $\cd_q(K,N)$-space for every $q \in (1,\infty)$ and for some $K \in \R, N\in(-\infty,0)$. If $K\ge 0$ or $K<0$ and diam$(\X) <\pi \sqrt{\frac {N-1}K}$, then $\Xdm$ is a BIP-space.
\end{corollary}
\begin{proof}
Observe that the hypotheses with Theorem \ref{thm:rajalaCDqKNegative} give the conclusion.
\end{proof}
\begin{corollary}
Let $\Xdm$ be a metric measure space that is a $\cd_2(K,N)$-space for some $N\in(-\infty,0)$ and $K \in \R$. Then, if $K\ge 0$ or $K<0$ and diam$(\X) <\pi \sqrt{\frac {N-1}K}$, $\Xdm$ supports a weak local $(1,1)$-Poincar\'e inequality.
\end{corollary}
\begin{proof}
This is a direct consequence of \cite[Theorem 4.1]{Rajala12-2} recalling Theorem \ref{thm:rajalaCDqKNegative}.
\end{proof}

\begin{lemma}\label{lem:identity}
It holds
\[ \lim_{n\to \infty} \prod_{i=1}^n \cos(2^{-i}\theta) = \frac{\sin(\theta)}{\theta}, \qquad \text{pointwise}.\]
\end{lemma}
\begin{proof}
Recalling the identity $\cos(2^{-i}\theta) = \tfrac{1}{2} \frac{\sin(2^{-i+1}\theta)}{\sin(2^{-i}\theta)}$, we have for every $n \in \N$: 
\[ \prod_{i=1}^n \cos(2^{-n}\theta) = \frac{1}{2^n}\prod_{i=1}^n \frac{\sin(2^{-i+1}\theta)}{\sin(2^{-1}\theta)} =  \frac{2^{-n}\theta}{\theta}\frac{\sin(\theta)}{\sin(2^{-n}\theta) }. \]
The claim follows simply by taking the limit as $n$ goes to infinity.
\end{proof}
\appendix
\section{Detecting the Sobolev space with a single test plan}\label{A}
This appendix is devoted to the study of master test plans on metric measure spaces, i.e. test plans that are capable to detect the Sobolev space and weak upper gradients. This notion has been the main object of the study in \cite{Pasqualetto20}, where the author asked whether this special object exists in \cite[Problem 2.7]{Pasqualetto20}. We will perform this analysis first on arbitrary metric measure spaces to provide a positive answer to this problem and then move to BIP-spaces where we can actually achieve a more sophisticated result.

Finally, we mention the closely related article \cite{NobPasSchu21} where a similar investigation has been conducted for the space of $BV$ functions.

\subsection{Master test plans on arbitrary metric measure spaces}
Let us start by defining the main object of this Appendix.
\begin{definition}[Master $q$-test plan]\label{A1}
Let $\Xdm$ be a metric measure space and $q\in(1,\infty)$. A \emph{master }$q$\emph{-test plan} $\ppi_q$ is a $q$-test plan so that:

if $f \colon \X \to \R$ Borel and $G\in L^p(\mm)$ with $G\ge 0$ are so that
\[ |f(\gamma_1)-f(\gamma_0)| \le \int_0^1 G(\gamma_t)\msp \, \d t,\qquad \ppi_q\text{-a.e. }\gamma,\]
then $f \in S^p(\X)$ and $G$ is a $p$-weak upper gradient.
\end{definition}
We point out that the definition is given differently from the original one in \cite[Definition 2.5]{Pasqualetto20} where the function $f$ is assumed to be Sobolev. The main reason is that, differently from there, we are going to prove that master test plans are also able to detect the full Sobolev space and not only the minimal weak upper gradients.

We now come to the first main result of this part giving a positive answer to \cite[Problem 2.7]{Pasqualetto20}.
\begin{theorem}\label{thm:Xdmpiq}
Let $\Xdm$ be a metric measure space. Then, for every $q \in (1,\infty)$, there exists a master $q$-test plan $\ppi_{q}$.
\end{theorem}
\begin{proof}
We fix an arbitrary $q \in (1,\infty)$ and subdivide the proof in different steps.

\noindent\textsc{Step 1.} Let us start defining for every $\alpha,\beta>0$ the set
\[ \Pi_{\alpha,\beta} := \{ \pi \in \cP(C([0,1],\X)) \colon \Comp(\pi)\le \alpha, \rmKe_q(\pi)\le \beta\}.\]
Recalling that $C([0,1],\X)$ with $\sfd_{\sup}$ is a complete and separable metric space, we have that $\cP(C([0,1],\X))$ is weakly separable. Therefore, we can consider for every $\alpha,\beta \in \Q$ countable and dense sets $\cA_{\alpha,\beta}\subset \Pi_{\alpha,\beta}$ and finally set
\[ \cA := \bigcup_{\alpha,\beta \in \Q} \cA_{\alpha,\beta}.\]
By construction, $\cA$ is countable.

\noindent\textsc{Step 2.} Let us then fix arbitrary $f$ Borel and $G \in L^p(\mm)$ with $G\ge 0$. We claim that $f \in S^p(\X)$ and $G$ is a $p$-weak upper gradient if and only if 
\begin{equation}
 \int |f(\gamma_0)-f(\gamma_1)|\, \d \pi \le \iint_0^1 G(\gamma_t)\msp\, \d t \d \pi,\qquad \forall \pi \in \cA.\label{eq:SobA}
\end{equation}
Obviously, we shall only prove the if-implication, as the converse is straightforward. To this aim, we fix an arbitrary $q$-test plan $\pi$ and consider sequences $(\alpha_n),(\beta_n)\subset \Q$ so that $\alpha_n \downarrow \Comp(\pi)$ and $\beta_n \downarrow \rmKe_q(\pi)$ as $n$ goes to infinity. Being $\pi \in \Pi_{\alpha_n,\beta_n}$ for all $n \in \N$, thanks to a diagonalization argument, we can find a sequence $\pi_n \in \cA$ so that
\[ \pi_n \weakto \pi, \qquad \text{and}\qquad \begin{array}{l}
\limsup_{n\to \infty}\rmKe_q(\pi_n) \le \rmKe_q(\pi),\\
\limsup_{n\to \infty}\Comp(\pi_n) \le \Comp(\pi).
\end{array}
\]
We observe that, passing the limit in \eqref{eq:SobA} with $\pi =\pi_n$ would give, given the arbitrariness of $\pi$, that $f \in S^p(\X)$ and $G$ is a $p$-weak upper gradient. To this aim, we invoke Lemma \ref{lem:limPol} to get that
\[ 
\lim_{n\to \infty}\iint_0^1G(\gamma_t)\msp\, \d t\d \pi_n = \iint_0^1G(\gamma_t)\msp\, \d t\d \pi.
\]
Now, arguing as in the proof of Proposition \ref{prop:sobolevPi}, it suffices to take to the limit the term $\int f(\gamma_1)-f(\gamma_0)|\, \d \pi_n$ without the absolute value inside to conclude. This can be done as the uniformly bounded compression of $(\pi_n)$ ensures that $\int f(\gamma_t)\, \d \pi_n \to \int f(\gamma_t)\, \d \pi$ as $n$ goes to infinity for every $t \in [0,1]$ when $f $ is bounded. For the general case, we argue with a truncation argument, again as in the proof of Proposition \ref{prop:sobolevPi}. 

\noindent\textsc{Step 3.} We now follow closely the strategy of \cite{Pasqualetto20} to reduce the countable collection $\cA$ to a single plan. We include all the details for completeness. Let then $(\pi_k)_{k \in \N}$ be an enumeration of $\cA$ and set
\[ \eta := \sum_{k\in \N} \frac{\pi_k}{2^{k}\max \{\Comp(\pi^{k}),\rmKe_q(\pi_k), 1\}}, \qquad \ppi_q := \frac{\eta}{\eta(C([0,1],\X)}.\]
The definition is well posed as $\eta(C([0,1],\X)) \le \sum_{n,l,i} 2^{-k} =1$. We claim that the so-defined plan is $q$-test plan. Indeed, it is by definition a probability measure on $C([0,1],\X)$ and, given any $t \in [0,1]$, the estimate
\[ (\e_t)_\sharp \eta \le \sum_{k \in \N}\frac{(\e_t)_\sharp\pi^{k}}{2^{k}\Comp(\pi^{k})} \le \sum_{n \in \N}2^{-k}\mm = \mm,\]
ensures that $\ppi_q$ has bounded compression. Moreover, we can estimate 
\[ 
\iint_0^1 \msp^q\, \d t\d \eta \le\sum_{k \in \N}\frac{1}{2^{k}\rmKe_q(\pi_k)}\iint_0^1 \msp^q \, \d t \d \pi_k = \sum_{k \in \N} 2^{-k} =1,\]
to show that  $\rmKe_q(\ppi) <\infty$ and consequently that $\ppi_q$ is a $q$-test plan.

\noindent\textsc{Step 4.} We conclude the proof by proving that $\ppi_q$ is a master $q$-test plan. By construction, for every $f$ Borel and $G \in L^p(\mm)$ with $G\ge 0$ we have that $|f(\gamma_1)-f(\gamma_0)|\le \int_0^1 G(\gamma_t)\msp \, \d t$ holds $\ppi_q$-a.e. if and only if it holds $\pi$-a.e. for every $\pi \in \cA$. Integrating then yields \eqref{eq:SobA} which in turn implies that $f \in S^p(\X)$ and $G$ is a $p$-weak upper gradient. The proof is now concluded.
\end{proof}

\subsection{Master test plan on BIP-spaces}
Here, we specialize the previous analysis to the context of BIP-spaces. In this case, we will also achieve that there exists a unique master test plan independent on $q$ that is also concentrated on  geodesics.

A special role here will be played by the set 
\[ {\sf Geod}(\X):= \cup_{q \in (1,\infty)} {\sf Geod}_q(\X)\]
which, recalling \eqref{eq:Geodq}, is a set of $\infty$-test plans thanks to the (BIP). Our first task is to reduce the class ${\sf Geod}(\X)$ given by the (BIP) to a countable number of plans, yet taking care that they are still capable of detecting the Sobolev space as in Proposition \ref{prop:sobolevPi}.
\begin{proposition}\label{A2}
Let $\Xdm$ be a BIP-space. Then, there exists a countable family $\mathcal D\subset {\sf Geod}(\X)$ of $\infty$-test plans concentrated on geodesics such that:

for every $p \in (1,\infty)$ and $f \colon \X \to \R$ Borel and $G \in L^p(\X)$ with $G\ge 0$, the following are equivalent 
\begin{itemize}
\item[{\rm i)}] $f \in S^p(\X)$ and $G$ is a $p$-weak upper gradient;

\item[{\rm ii)}]it holds \[ \int |f(\gamma_1)-f(\gamma_0)|\, \d \pi \le \iint_0^1 G(\gamma_t)\msp \, \d t\d \pi,\qquad \forall \pi \in\mathcal D.\]
\end{itemize}
\end{proposition}
\begin{proof}
We subdivide the proof in a reduction step and afterwards, we prove the equivalence.

\noindent\textsc{Reduction.} Let $\bar x \in \X$ be a point and consider, for every $k \in \N$, the set of plans
\[ \Pi_k := \big\{ \pi \in {\sf Geod}(\X) \colon \Comp(\pi) \le k,  \text{ supp}((\e_i)_\sharp\pi) \subseteq B_k(\bar x), \ i=0,1\big\}.\]
Fix $k \in \N$, any $\pi \in \Pi_k$ is concentrated on geodesics lying in $B_k(\bar x)$ and $(\e_t)_\sharp \pi \le k \mm\restr{B_k(\bar x)}$ for every $t \in [0,1]$. Hence, the family $\{(\e_t)_\sharp \pi \colon t \in [0,1],\pi \in \Pi_k\}$ is tight and, by Prokhorov's Theorem \ref{thm:prokh}, there exists a functional $\psi \colon \X \to \R$ with compact sublevels so that \[\sup_{\pi \in \Pi_k,t \in [0,1]}\int \psi \, \d (\e_t)_\sharp \pi<\infty.\] Then, arguing as in \cite[Lemma 5.8]{GMS15} (or as in Lemma \ref{lem:approxPol}) and recalling that only geodesics with uniformly bounded length are to be considered, we can consider lifting $\psi$ to the functional $\Psi \colon C([0,1],\X) \to \R$, defined via $\Psi(\gamma) := \int\psi(\gamma_t)\, \d t + \sfd(\gamma_0,\gamma_1)$ if $\gamma \in \Geo(\X)$ and $+\infty$ otherwise,  that has compact sublevels and satisfies
\[ \sup_{\pi \in \Pi_k} \int \Psi(\gamma)\, \d \pi < \infty.\]
Using again Prokhorov's Theorem \ref{thm:prokh}, we get that $\Pi_k$ is relative compact in the weak topology of $\cP(C([0,1],\X))$. Now, for every $k\in \N$, consider a countable and dense collection $\mathcal D_k\subset \Pi_k$ and lastly define
\begin{equation}
\mathcal D:= \bigcup_{k\in \N} \mathcal D_k \subseteq {\sf Geod}(\X). \label{eq:cPN}
\end{equation} 
It is then obvious by construction that the class $\mathcal D$ is a countable collection of $\infty$-test plans concentrated on geodesics.

\noindent\textsc{Equivalence.} The implication $(i)\Rightarrow (ii)$ is obvious. For the converse $(ii)\Rightarrow (i)$, we remark that it is sufficient to show 
\begin{equation} \int f(\gamma_1)-f(\gamma_0)\, \d \pi \le \iint_0^1 G(\gamma_t)\msp \, \d t\d \pi,\qquad \forall \pi \in {\sf Geod}(\X),\label{eq:Ainter}
\end{equation}
as the conclusion will then follow invoking Proposition \ref{prop:sobolevPi} and arguing as in the proof of Proposition \ref{prop:sobolevPi} to improve from the above inequality to the one having $\int| f(\gamma_1)-f(\gamma_0)|\, \d \pi$ at the left hand side. Then, we pick $\pi \in {\sf Geod}(\X)$ and observe that there exists $k \in \N$ so that $\pi \in \Pi_k$. Then, consider a sequence $(\pi_n)\subseteq \mathcal D$ so that $\pi_n\weakto \pi$ as $n$ goes to infinity and, by construction, we can take in $\mathcal D_k$ for a suitable $k$). Then, the hypotheses ensure that
\begin{equation}
\int f(\gamma_1)-f(\gamma_0)\, \d \pi^n (\gamma)\le \iint_0^1 G(\gamma_t)\msp\, \d t\d \pi^n(\gamma)=\iint_0^1 G(\gamma_t)\sfd(\gamma_0,\gamma_1)\, \d t\d \pi^n(\gamma),\qquad \forall n \in \N,\label{A:ineq1}
\end{equation} 
having used the fact that $\pi_n$ is concentrated on geodesics in the last step. Since the function $\gamma\mapsto\sfd(\gamma_0,\gamma_1)$ is continuous and bounded on bounded sets, the plans $(\sfd(\gamma_0,\gamma_1)\,\d \pi_n(\gamma))$ weakly converge to $\sfd(\gamma_0,\gamma_1)\,\d \pi(\gamma)$. Since clearly they have uniformly bounded compression, by arguing as in the proof of Proposition \ref{prop:sobolevPi} we see that
\[
\iint_0^1 G(\gamma_t)\sfd(\gamma_0,\gamma_1)\, \d t\d \pi_n(\gamma)\quad\to\quad \iint_0^1 G(\gamma_t)\sfd(\gamma_0,\gamma_1)\, \d t\d \pi(\gamma)
\]
To pass to the limit in the left hand side of \eqref{A:ineq1}  we can argue e.g. as in the proof of Proposition \ref{prop:sobolevPi}, again using the assumption of uniformly bounded compression. Finally, we achieved \eqref{eq:Ainter} and the conclusion.
\end{proof}
Mimicking an argument in \cite{Pasqualetto20},  we can pass from a countable collection of plans detecting the minimal $p$-weak upper gradient to just one.

\begin{theorem}\label{A3}
Let $\Xdm$ be a BIP-space. Then, there exists a $\infty$-test plan $\ppi_{\sf master}$ concentrated on geodesics so that 
\[ \ppi_{\sf master}\text{ is a master }q\text{-plan},\qquad \forall q \in (1,\infty).\]
\end{theorem}
\begin{proof} Let $\mathcal D$ be given by Proposition \ref{A2} and $(\pi^n)$ an enumeration of the countable collection
\[
\mathcal C:=\big\{\big( { \sf Restr}_{\frac{i-1}{k}}^{\frac{i}{k}}\big)_\sharp \pi\ :\ k\geq \|\Lip(\gamma)\|_{L^\infty(\pi)},\ k\in\N,\ i=1,...,k,\ \pi\in\mathcal D\big\}\subset\cP(C([0,1],\X))
\]
and define
\[
 \eta := \sum_{{ n\in \N}} \frac{\pi^n}{2^{n}\max \{\Comp(\pi^n), 1\}}, \qquad\qquad \ppi_{\sf master} := \frac{\eta}{\eta(C([0,1],\X)}. 
\]
The definition is well posed as $\eta(C([0,1],\X)) \le \sum_{n} 2^n <\infty$. We claim that $\ppi_{\sf master}\in\cP(C([0,1],\X))$ satisfies the requirements and we start checking that it is a $\infty$-test plan.

For $t\in[0,1]$ we have
\[
(\e_t)_\sharp \pi\leq \frac1{\eta(C([0,1],\X)} \sum_{{ n\in \N}} \frac{(\e_t)_\sharp\pi^n}{2^{n}\Comp(\pi^n)}\leq \frac1{\eta(C([0,1],\X)} \sum_{{ n\in \N}} \frac{\mm}{2^{n}}=\frac{\mm}{\eta(C([0,1],\X)},
\]
and thus $\ppi_{\sf master}$ has bounded compression. Also, since every element $\pi^n$ of $\mathcal C$ is such that  $\|\Lip(\gamma)\|_{L^\infty(\pi^n)} \le 2$, we have that $\|\Lip(\gamma)\|_{L^\infty(\ppi_{\sf master})} \le 2$ as well. We thus proved that   $\ppi_{\sf master}$ is a $\infty$-test plan.

Now let  $p \in (1,\infty)$,$f$ Borel and $G\in L^p(\mm)$ with $G\ge0$, be such that
\[ 
|f(\gamma_1)-f(\gamma_0)| \le \int_0^1 G(\gamma_t)\msp \, \d t,\qquad \ppi_{\sf master}\text{-a.e. }\gamma,
\]
and notice that since by construction $ \ppi_{\sf master}$-negligible sets are also $\pi$-negligible for any $\pi\in\mathcal C$, we have
\[ 
\int |f(\gamma_1)-f(\gamma_0)| \,\d\pi(\gamma)\le \iint_0^1 G(\gamma_t)\msp \, \d t\,\d \pi(\gamma)
\]
for every $\pi\in\mathcal C$. By definition of $\mathcal C$ and a simple gluing argument it is then clear that this last inequality holds for any $\pi\in\mathcal D$, hence the conclusion follows from Proposition \ref{A2}.
\end{proof}
\begin{remark}
\rm
We point out a key remark from \cite{Pasqualetto20} which is worth to notice also in this note. The results contained in Theorem \ref{thm:Xdmpiq} and Theorem \ref{A3} are not just technical. Indeed, the existence of master test plans as in Definition \ref{A1} make it possible to identify which are the exceptional curves for which the weak upper gradient inequality $|f(\gamma_1)-f(
\gamma_0)|\leq\int_0^1|Df|_p(\gamma_t)|\dot\gamma_t|\,\d t$ fails. This could be previously done by appealing to the notion of $q$-Modulus from \cite{Koskela-MacManus98} and further employed  in \cite{Shanmugalingam00} for a systematic definition of Sobolev space. The key difference is that the $q$-Modulus is only an outer measure, while master $q$-test plans $\ppi_q$ (or, on BIP-spaces, the more powerful $\ppi_{\sf master}$ plan) are Borel probability measures.
\fr
\end{remark}

\section{Proof of Theorem \ref{thm:rajalaCDqKN}}\label{B}
The content of this Appendix is to prove Theorem \ref{thm:rajalaCDqKN}. As previously remarked, the proof goes along the same lines of \cite{Rajala12-2} (where the case of $\cd_2$-spaces is treated) hence, we shall adopt a concise style in the following to handle the case of $\cd_q$-spaces for arbitrary $q \in (1,\infty)$. Next, we denote
\[ 
C(D,K):= e^{K^-D^2/12},
\]
and recall the statement of Theorem \ref{thm:rajalaCDqKN}.
\begin{theorem}\label{B_CDqN}
Let $\Xdm$ be a $\cd_q(K,\infty)$-space for some $K \in \R$ and $D> 0$. For any $\rho_0 ,\rho_1 \in L^{\infty}(\mm)$ probability densities with diam$($supp$(\rho_0) \ \cup\ $supp$(\rho_1))<D$, there exists $\pi \in \OptGeo_q(\rho_0\mm,\rho_1\mm)$ satisfying $\mu_t:=(\e_t)_\sharp \pi \ll \mm$. Moreover, writing $\mu_t:= \rho_t\mm$, we have the following upper bound for the density
\[ \|\rho_t\|_{L^\infty(\mm)} \le C(D,K) \|\rho_0\|_{L^\infty(\mm)}\vee \|\rho_1\|_{L^\infty(\mm)},\qquad \forall t \in [0,1].\]
\end{theorem}

We face, before the proof, the necessary preparatory results. We shall adopt a concise presentation in the sequel, as the arguments in this Appendix are taken from \cite{Rajala12-2} for the case $q=2$ and applies with minor modifications in our context.

\subsubsection*{Preparatory Lemmas}
Consider for any $q \in (1.\infty)$ two measures $\mu_0,\mu_1 \in \cP_q(\X)$ with $W_q(\mu_0,\mu_1)<\infty$ and denote by 
\[ \cI^q_t(\mu_0,\mu_1) := \{\mu \in \cP_q(\X) \colon W_q(\mu_0,\mu) = tW_q(\mu_0,\mu_1), \ W_q(\mu,\mu_1) = (1-t)W_q(\mu_0,\mu_1)\},\]
the set of $t$-intermediate measures between $\mu_0,\mu_1$ where $t \in (0,1)$. 
\begin{lemma}\label{lem:Itqclosed}
Let $(\X,\sfd)$ be a metric space, $q \in (1,\infty)$ and assume $\mu_0,\mu_1 \in \cP_q(\X)$ have bounded supports. Then, for every $t \in (0,1)$, the set $\cI_t^q(\mu_0,\mu_1)$ is closed in $(\cP_q(\X),W_q)$.
\end{lemma}
\begin{proof}
For any $\nu_n \subseteq \cI_t^q(\mu_0,\mu_1)$ with $\nu_n \rightarrow \nu$ in $W_q$, the triangular inequality gives
\[ \max_{i=0,1} |W_q(\mu_i,\nu)-W_q(\mu_i,\nu_n)| \le W_q(\nu_n,\nu), \]
from which the conclusion follows.
\end{proof}

We now face convexity properties of the set of $t$-intermediate measures. This statement should be interpreted as a way to redistribute mass on intermediate points of Wasserstein geodesics.
\begin{lemma}\label{lem:redistrmass}
Let $\Xdm$ be a metric measure space which is also geodesic and $q \in (1,\infty)$. Suppose $\mu_0, \mu_1 \in \cP(\X)$ with $W_q(\mu_0,\mu_1)<\infty$. Then, for any $\pi \in \OptGeo_q(\mu_0,\mu_1)$ and $f \colon \Geo(\X) \rightarrow [0,1]$ s.t. $c=(f\pi)(\Geo(\X))\in (0,1)$, we have
\[ (\e_t)_\sharp((1-f)\pi)+c\mu \in \cI_t^q(\mu_0,\mu_1),\]
for every $ \mu \in \cI_t^q(\frac{1}{c}(\e_0)_\sharp(f\pi),\frac1c(\e_1)_\sharp(f\pi), \ t \in (0,1)$.
\end{lemma}
\begin{proof}
We omit the details and refer to \cite[Lemma 3.5]{Rajala12-2} for the proof which reads identical for arbitrary $q$.
\end{proof}
We consider the excess mass functional
\[ \cF_C(\eta) := \|(\rho-C)^+\|_{L^1(\mm)}+ \eta^s(\X),\]
and observe that, when it vanishes on a probability measure, it automatically detects both absolutely continuity with respect to the reference measure with corresponding density $L^\infty$-bounded from the constant $C$. We prove first that it is lower semicontinuous.
\begin{lemma}\label{lem:FClsc}
Let $(\X,\sfd)$ be a bounded metric space equipped with a finite Borel measure $\mm$ and $q \in (1,\infty)$. Then, for any $C\ge 0$, the functional $\cF_C$ is lower semicontinuous on $(\cP_q(\X),W_q)$.
\end{lemma}
\begin{proof}
Rajala's proof consists in showing that
\[ \cF_C(\mu) = \sup \big\{\int g\, \d \mu -C\int g\, \d \mm \colon g \in C(\X), \ 0\le g\le 1 \big\},\]
for every $\mu \in \cP_q(\X)$. It is proven in \cite{Rajala12-2}) that it is lower semicontinuous in $\PP_2(\X)$. Then, recalling \eqref{eq:Wqweak} and the present hypothesis, $\cF_C$ is equivalently lower semicontinuous in the space $(\PP_q(\X),W_q)$ independently on $q \in (1,\infty)$.
\end{proof}

\subsubsection*{Proof of Theorem \ref{B_CDqN} (and Theorem \ref{thm:rajalaCDqKN})} 
We subdivide the proof into several steps. Write for simplicity $ \mu_i:= \rho_i\mm$ for $i=0,1$. First observe that, for every $K$, we can alternatively prove the statement in the larger (or equal) class of $\cd_q(-K^-,\infty)$-spaces. Hence, for simplicity, we shall consider only $K<0$ in the proof. 

\noindent\textsc{Step 1}. From the $\cd_q(K,\infty)$ condition, there exists $\pi \in \OptGeo_q(\mu_0,\mu_1)$ which is concentrated on geodesics with length at most $D$. Also, we denote for simplicity $E:= \{\rho_\frac12>0\}$. We observe firstly that Jensen's inequality with the convexity of $u(x)=x\log x$ guarantees
\[ 
{\rm Ent}_\mm(\mu_\frac12) = \int _E \rho_\frac12\log\rho_\frac12 \, \d \mm  \ge \log \Big( \frac{1}{\mm(E)}\Big),
\]
and secondly that 
\[\begin{split}
{\rm Ent}_\mm(\mu_\frac12) &\le \frac12{\rm Ent}_\mm(\mu_0)+\frac12{\rm Ent}_\mm(\mu_1) +\frac{-K}8W_q^2(\mu_0,\mu_1) \\
&\le \log\big(\|\rho_0\|_{L^\infty(\mm)}\vee \|\rho_1\|_{L^\infty(\mm)}\big) +\frac{-KD^2}8.
\end{split} \]
We can thus combine in both case the two inequalities to get the following \emph{spreading of mass under curvature dimension condition} principle:
\begin{equation}
\mm(\{\rho_\frac12>0\}) \ge \frac{1}{P(D,K)\|\rho_0\|_{L^\infty(\mm)}\vee \|\rho_1\|_{L^\infty(\mm)} }, \label{eq:spreadmass}
\end{equation}  
where $P(D,K)= e^{K^-D^2/8}$. For simplicity, we write from now on $M:= P(D,K)\|\rho_0\|_{L^\infty(\mm)}\vee \|\rho_1\|_{L^\infty(\mm)} $.

\noindent\textsc{Step 2}. We claim that for any $C>M$, there exists a minimizer of $ \cF_C(\cdot)$ in $\cI_\frac12^q(\mu_0,\mu_1)$.   This has been shown in \cite{Rajala12-2} building exactly upon \eqref{eq:spreadmass}, Lemma \ref{lem:Itqclosed}, Lemma \ref{lem:redistrmass} and Lemma \ref{lem:FClsc} for the particular case $q=2$. Being these results valid also for general $q \in (1,\infty)$ we omit the details of the strategy ensuring that 
\begin{equation}
\forall C> M, \ \exists \mu \in \cI_\frac12^q(\mu_0,\mu_1) \quad \text{ so that } \quad \cF_C(\mu) = \inf_{\eta \in \cI_\frac12^q(\mu_0,\mu_1)} \cF_C(\eta).\label{eq:inter1}
\end{equation}

\noindent\textsc{Step 3}.  For any $C > M$, we claim that
\[ \inf_{\eta \in \cI_\frac12^q(\mu_0,\mu_1)} \cF_C(\eta)=0.\]
Denote $\cI^q_{min}\subset \cI_\frac12^q(\mu_0,\mu_1)$ the set of minimizers of $\cF_C$ (which is always nonempty \eqref{eq:inter1}) and let $\mu \in \cI^q_{min}$ be such that 
\begin{equation} \mm(\rho_\mu>C) \ge \Big(\frac MC\Big)^{\frac14} \sup_{\eta \in \cI_\frac12^q(\mu_0,\mu_1)} \mm(\{\rho_\eta>C\}), \label{eq:Binterim}
\end{equation}
where $\mu:= \rho_\mu\mm+\mu^s$ with $\mu^s\perp\mm$ and $\eta:= \rho_\eta\mm+\mu^s$ with $\eta^s\perp \mm$. We argue now by contradiction and suppose $\cF_C(\mu)>0$ whence. If $A:=\{\rho_\mu>0\}$, then this means necessarily that $\mm(A)>0$ and $\mu^2(\X) >0$. In the first case, find a $\delta >0$ so that, denoting $A':=\{\rho_\mu>C+\delta\}$, we have 
\[ \mm(A') \ge \Big(\frac{M}{C}\Big)^{\frac12}\mm(A).\]
Let $\alpha \in \Opt_q(\mu_0,\mu), \beta \in \Opt_q(\mu,\mu_1)$ and consider 
\[ \tilde{\pi} \in \OptGeo_q\Big( (P^0)_\sharp\frac{\alpha\restr{\X\times A'}}{\mu(A')}, (P^1)_\sharp\frac{\beta\restr{A'\times \X}}{\mu(A')}  \Big),\]
given by Step 1. Denote $\Gamma_t := (\e_t)_\sharp\tilde{\pi}$ the corresponding $W_q$-geodesic and consider its decomposition $\Gamma_\frac12 = \rho_\Gamma\mm +\Gamma^s$. Then, from \eqref{eq:spreadmass}, it follows that
\begin{equation}
 \mm(\{\rho_\Gamma >0\} \ge \frac{\mu(A')}{M} \ge \frac{C}{M}\mm(A')\ge\Big(\frac{C}{M}\Big)^\frac12\mm(A). \label{eq:Binterim2}
 \end{equation}
Now  consider redistributing the mass of the measure $\mu$ via
\[\tilde{\mu}:= \mu\restr{X\setminus A'} + \frac{C}{C+\delta}\mu\restr{A'}+\frac\delta{C+\delta}\mu(A')\Gamma_\frac12. \]

Arguing then as in \cite[Lemma 3.5]{Rajala12-2}, Lemma \ref{lem:redistrmass} directly yields $\tilde{\mu} \in \cI_\frac12^q(\mu_0,\mu_1)$. Also, setting $\tilde{\mu}=\rho_{\tilde{\mu}}\mm+\tilde{\mu}^s$ with $\tilde{\mu}^s\perp \mm$, a standard calculation shows that the excess functional decreases. We omit here the details to get 
\[ \cF_C(\mu)-\cF_C(\tilde{\mu}) = \int_{\{\rho_\mu <C\}} \min\big\{ C-\rho_\mu, \frac{\delta}{C+\delta}\mu(A')\rho_\Gamma\Big\}\, \d \mm. \]
Notice that, being $\mu$ a minimizer, the right hand most side of the above equation is nonpositive whence, necessarily the integral must vanish giving in turn 
\[ \tilde{\mu} \in \cI_{min}^q(\mu_0,\mu_1), \qquad \mm(\{\rho_\mu<C\}\cap \{\rho_\Gamma>0\})=0.\]
Moreover, $\rho_{\tilde{\mu}}>C$ $\mm$-a.e. on the set $\{\rho_\mu \ge C\}\cap \{\rho_\Gamma>0\}$ , hence 
\[\mm (\{ \rho_{\tilde{\mu}}>C\} ) \ge \mm(\{\rho_\Gamma>0\}) \overset{\eqref{eq:Binterim2}}{\ge} \Big(\frac{C}{M}\Big)^\frac12\mm(\{ \rho_\mu >C\} ) \overset{\eqref{eq:Binterim}}{\ge}  \Big(\frac CM\Big)^{\frac14} \sup_{\eta \in \cI_\frac12^q(\mu_0,\mu_1)} \mm(\{\rho_\eta>C\}), \]
yielding a contradiction, since $\tilde{\mu} \in \cI_\frac12^q(\mu_0,\mu_1)$ and $C>M$. Therefore, $A$ is negligible and the first situation does not occur: necessarily $\mu$ is purely singular, otherwise there is nothing to prove. But then a similar redistribution of mass for the singular part applies giving a contradiction. Wrapping up, we showed
\[ \forall C>M \qquad \exists \min_{\eta \in \cI_\frac12^q(\mu_0,\mu_1)}\cF_C(\eta) =0.\]
The extreme case $C=M$ can be obtained with an easy argument as in \cite[Corollary 3.12]{Rajala12-2}: being $\mu_0,\mu_1$ supported on bounded sets, we can find a bounded set $B\subset \X$ so that every $\eta \in \cI_\frac12^q(\mu_0,\mu_1)$ is supported in $B$, and hence
\[ \min_{\eta \in \cI_\frac12^q(\mu_0,\mu_1)}\cF_M(\eta) \le \min_{\eta \in \cI_\frac12^q(\mu_0,\mu_1)}\cF_C(\eta) + |C-M|\mm(B) = |C-M |\mm(B),\]
for every $C>M$. Thus, the conclusion follows also for $C=M$ by approximation $C\downarrow M$.

\noindent\textsc{Step 4}. The conclusion of the theorem will be achieved by iterating the above construction from midpoint to a general dyadic partition of $[0,1]$. Fix $n\in \N$, we now show how to produce from the measures $(\rho_{k2^{-n+1}})$ for $k =0,...,2^{-n+1}$ the successive sequence $(\rho_{k2^{-kn}})$. Consider, for $k$ odd, the midpoints $\mu_{k2^{-n}} \in \cI_\frac12^q(\rho_{k2^{-n+1}}\mm,\rho_{(k+1)2^{-n+1}}\mm)$ satisfying
\[
\begin{split}
 &\mu_{k2^{-n}} \ll \mm, \qquad \mu_{k2^{-n}}:=\rho_{k2^{-n}} \mm \\
 &\|\rho_{k2^{-n}} \|_{L^\infty(\mm)} \le  P(2^{-n+1}D,K)\|\rho_{(k-1)2^{-n}} \|_{L^\infty(\mm)}\vee \|\rho_{(k+1)2^{-n}} \|_{L^\infty(\mm)},
 \end{split}
  \]
since diam$($supp$(\rho_{(k-1)2^{-n}}) \cup$supp$(\rho_{(k+1)2^{-n})})< 2^{-n+1}D$ and all the previous steps apply. By induction, it is easy to prove that
\[
\|\rho_{k2^{-n}} \|_{L^\infty(\mm)} \le \prod_{i=1}^n P(2^{-i+1}D,K)\|\rho_0 \|_{L^\infty(\mm)}\vee \|\rho_0 \|_{L^\infty(\mm)}, \qquad \forall n \in \N,
\]
which, under the assumption $D<\infty$, can be coupled with the fact 
\[ C(D,K):=\lim_{n\rightarrow \infty}  \prod_{i=1}^n P(2^{-i+1}D,K) <\infty, \]
giving in turn that a geodesic curve $\mu_t$ in $\OptGeo_q(\mu_0,\mu_1)$ is well defined by completion. Finally, the $L^\infty$-estimate on $\rho_t$ holds by lower semicontinuity of the functional $\cF_M$.
\qed

\bigskip{\bf Acknowledgment.} We wish to thank F. Cavalletti for his valuable comments around the topics of this note and E. Pasqualetto for useful remarks concerning Appendix \ref{A} and for suggesting Step 1 of Theorem \ref{thm:BIPMCP}.

We also warmly thank the referees for their useful comments.

\def\cprime{$'$} \def\cprime{$'$}

\end{document}